\documentclass[11pt]{amsart}
\usepackage{amssymb,amsthm,amsmath,amstext}
\usepackage{mathrsfs}  
\usepackage{bm}        
\usepackage{mathtools} 
\usepackage{color}
\usepackage[bbgreekl]{mathbbol}

\usepackage[left=1.5in,top=1in,right=1.5in,bottom=1in]{geometry}

\usepackage{graphicx}

\usepackage[all]{xy}

\theoremstyle{plain}
\newtheorem{theorem}{Theorem}[section]

\newtheorem{proposition}[theorem]{Proposition}
\newtheorem{lemma}[theorem]{Lemma}
\newtheorem{corollary}[theorem]{Corollary}

\theoremstyle{definition}
\newtheorem{definition}[theorem]{Definition}
\newtheorem{example}[theorem]{Example}

\theoremstyle{remark}
\newtheorem{remark}[theorem]{Remark}

\newcommand{\sheaf}[1]{\mathscr{#1}}
\newcommand{\LL}{\sheaf{L}}
\newcommand{\OO}{\sheaf{O}}

\newcommand{\EE}{\sheaf{E}}

\newcommand{\BB}{\sheaf{B}}

\newcommand{\RR}{\sheaf{R}}
\newcommand{\sAA}{\sheaf{A}}

\newcommand{\valus}[1]{\mathsf{#1}}

\newcommand{\F}{\mathbb F}
\newcommand{\Z}{\mathbb Z}
\newcommand{\A}{\mathbb A}
\newcommand{\C}{\mathbb C}
\renewcommand{\P}{\mathbb P}
\newcommand{\Q}{\mathbb Q}
\newcommand{\Gm}{\mathbb{G}_{\mathrm{m}}}

\DeclareMathOperator{\Gal}{\mathrm{Gal}}
\DeclareMathOperator{\Br}{\mathrm{Br}}

\newcommand{\inv}{^{-1}}
\newcommand{\sep}{^\mathrm{s}}
\newcommand{\mult}{^{\times}}
\newcommand{\multtwo}{^{\times 2}}

\newcommand{\tensor}{\otimes}

\newcommand{\ol}[1]{\overline{#1}}

\newcommand{\wt}[1]{\widetilde{#1}}
\newcommand{\et}{\mathrm{\acute{e}t}}
\newcommand{\univ}{\mathrm{univ}}

\newcommand{\Het}{H_{\et}}
\newcommand{\ur}{\mathrm{nr}}

\newcommand{\CH}{\mathrm{CH}}

\usepackage[backref=page]{hyperref}
\hypersetup{pdftitle={Unramified Brauer groups of conic bundle threefolds in characteristic two}}
\hypersetup{pdfauthor={Asher Auel, Alessandro Bigazzi, Christian Boehning, Hans-Christian Graf von Bothmer}}
\hypersetup{colorlinks=true,linkcolor=blue,anchorcolor=blue,citecolor=blue,letterpaper=true}

\begin{document}

\title[Unramified Brauer groups in characteristic two]{Unramified
Brauer groups of conic bundle threefolds in characteristic two}
\author[Auel]{Asher Auel}
\address{Asher Auel, Department of Mathematics, Yale University\\
New Haven, Connecticut 06511, USA}
\email{asher.auel@yale.edu}

\author[Bigazzi]{Alessandro Bigazzi}
\address{Alessandro Bigazzi, Mathematics Institute, University of Warwick\\
Coventry CV4 7AL, England}
\email{A.Bigazzi@warwick.ac.uk}

\author[B\"ohning]{Christian B\"ohning}
\address{Christian B\"ohning, Mathematics Institute, University of Warwick\\
Coventry CV4 7AL, England}
\email{C.Boehning@warwick.ac.uk}

\author[Bothmer]{Hans-Christian Graf von Bothmer}
\address{Hans-Christian Graf von Bothmer, Fachbereich Mathematik der Universit\"at Hamburg\\
Bundesstra\ss e 55\\
20146 Hamburg, Germany}
\email{hans.christian.v.bothmer@uni-hamburg.de}


\date{\today}


\begin{abstract}
We establish a formula for computing the unramified Brauer group of
tame conic bundle threefolds in characteristic $2$. The formula
depends on the arrangement and residue double covers of the
discriminant components, the latter being governed by Artin--Schreier
theory (instead of Kummer theory in characteristic not $2$). We use
this to give new examples of threefold conic bundles defined over $\Z$
that are not stably rational over $\C$.
\end{abstract}

\maketitle

\section{Introduction}

Motivated by the rationality problem in algebraic geometry, we compute
new obstructions to the universal triviality of the Chow group of
$0$-cycles of smooth projective varieties in characteristic $p>0$,
related to ideas coming from crystalline cohomology (see \cite{CL98}
for a survey): the $p$-torsion in the unramified Brauer group.  In
\cite{ABBB18}, we proved that $p$-torsion Brauer classes do obstruct
the universal triviality of the Chow group of $0$-cycles; here we
focus on computing these obstructions in the case of conic bundles. In
particular, we provide a formula to compute the two torsion in the
unramified Brauer group of conic bundle threefolds in characteristic~2. We provide some applications showing that one can obtain results
with this type of obstruction that one cannot by other means: there
exist conic bundles over $\P^2$, defined over $\Z$, that are smooth
over $\Q$ and whose reduction modulo $p$ has (1) nontrivial two
torsion in the unramified Brauer group and a universally
$\CH_0$-trivial resolution for $p=2$, and (2) irreducible
discriminant, hence trivial Brauer group, for all $p > 2$.  The
roadmap for this paper is as follows.

In Section \ref{sConicBundles}, we assemble some background on conic
bundles and quadratic forms in characteristic $2$ to fix notation and
the basic notions.

Section \ref{sBrauerMerkurjev} contains a few preliminary results
about Brauer groups, in particular their $p$-parts in characteristic
$p$, and then goes on to discuss residue maps, which, in our setting,
are only defined on a certain subgroup of the Brauer group, the
so-called tamely ramified, or tame, subgroup. We also interpret these
residues geometrically for Brauer classes induced by conic bundles in
characteristic $2$, and show that their vanishing characterizes
unramified elements.

Besides the fact that residues are only partially defined, we
encounter another new phenomenon in the bad torsion setting, which we
discuss in Section~\ref{sDiscriminantsChar2}, namely, that Bloch--Ogus
type complexes fail to explain which residue profiles are actually
realized by Brauer classes on the base. We also investigate some local
analytic normal forms of the discriminants of conic bundles in
characteristic $2$ that can arise or are excluded for various reasons.

In Section \ref{sFormulaBrauerChar2} we prove
Theorem~\ref{tBrauerNontrivial}, which computes the two torsion in the
unramified Brauer group of some conic bundles over surfaces in
characteristic two. In the hypotheses, we have to assume the existence
of certain auxiliary conic bundles over the base with predefined
residue subprofiles of the discriminant profile of the initial conic
bundle. This is because of the absence of Bloch--Ogus complex methods,
as discussed in Section~\ref{sDiscriminantsChar2}.

In Section \ref{sExamplesNontrivialBrauerChar2} we construct examples
of conic bundle threefolds defined over $\Z$ that are not stably
rational over $\C$, of the type described in the first paragraph of
this Introduction.

{\bf Acknowledgments.} We would like to thank Jean-Louis Colliot-Th\'{e}l\`{e}ne and Burt Totaro for their interest in this work, their helpful suggestions and pointers to the literature. We would also like to thank the Simons Foundation for hosting the Conference on Birational Geometry August 2017 during which part of this work was presented and several participants gave us useful feedback.

\section{Background on conic bundles}\label{sConicBundles}

Let $K$ be a field. The most interesting case for us in the sequel
will be when $K$ has characteristic $2$. Typically, $K$ will not be
algebraically closed, for example, the function field of some
positive-dimensional algebraic variety over an algebraically closed
ground field $k$ of characteristic $2$, which is the base space of
certain conic bundles or, more generally, quadric fibrations.

\subsection{Quadratic forms.}

As a matter of reference, let us recall here some basic notions
concerning the classification of quadratic forms in characteristic
2. We refer to \cite[Ch.~I--II]{EKM08}.

\begin{definition} \label{quadr_form_general}
Let $V$ be a finite dimensional vector space over $K$. A \emph{quadratic
form} over $V$ is a map $q:V\longrightarrow K$ such that:
\begin{enumerate}
\item $q(\lambda v)=\lambda^2 q(v)$ for each $\lambda\in K$ and $v\in V$;
\item the map $b_q : V\times V \longrightarrow K$ defined by
\[
b_q (v,w)=q(v+w)-q(v)-q(w)
\]
is $K$-bilinear.
\end{enumerate}
\end{definition}

When the characteristic of $K$ is not $2$, a quadratic form
$q$ can be completely recovered by its associated bilinear form $b_q$
and, thus, by its associated symmetric matrix. This correspondence 
fails to hold when $\mathrm{char}\,K=2$, due to the existence of non-zero
quadratic forms with identically zero associated bilinear form; these
forms are called \emph{totally singular} and play a significant role
in the decomposition of quadratic forms over such fields.

\begin{definition} \label{quadr_form_radical}
Let $b$ be a bilinear form over $V$; its \emph{radical} is the set
\[
r(b):=\{v\in V\mid b(v,w)=0\text{ for any }w\in V\}
\]
Let $q$ be a quadratic form; the \emph{quadratic radical} is
\[
r(q):=\{v\in V\mid q(v)=0\}\cap r(b_q)
\]
In general, we have strict inclusion $r(q)\subset r(b_q)$ if
$\mathrm{char}\,K=2$.  A form such that $r(q)=0$ is called
\emph{regular}.
\end{definition}

We introduce the following notation: let $q$ be a quadratic form
over $V$ and let $U,W\subseteq V$ be vector subspaces such that
$V=W\oplus U$. If $U$ and $W$ are orthogonal with respect to the
associated bilinear form $b_q$ (we write $U\subset W^{\bot}$ to
mean this), then $q$ decomposes as sum of its restrictions $q|_W$
and $q|_U$ and we write $q=q|_W\bot q|_V$.

We will also say that two quadratic forms $q_1,q_2$ defined
respectively over $V_1$ and $V_2$, are \emph{isometric} if there
exists an isometry $f:V_1\longrightarrow V_2$ of the associated
bilinear forms and satisfying $q_1(v)=q_2(f(v))$. In this case, we
write $q_1\simeq q_2$.  We say that $q_1, q_2$ are \emph{similar} if
$q_1 \simeq c q_2$ for some $c \in K\mult$.

\begin{definition} \label{quadr_form_diag_anisotr}
Let $a,b\in K$. We denote by $\langle a\rangle$
the \emph{diagonal quadratic form} on $K$ (as $K$-vector space over
itself) defined by $v\mapsto a\, v^2$. Also, we denote by $[a,b]$
the quadratic form on $K^2$ defined by $(x,y)\mapsto ax^2+xy+by^2$.
\end{definition}

We say that a quadratic form $q$ is \emph{diagonalizable} if there
exists a direct sum decomposition $V=V_1\oplus\ldots\oplus V_n$
such that each $V_i$ has dimension $1$, we have $V_i\subseteq V_j^{\bot}$
for every $i\neq j$ and $q|_{V_i}\simeq\langle a_i\rangle$ so
that
\[
q\simeq\langle a_1,\ldots,a_n\rangle:=\langle a_1\rangle\bot\ldots\bot\langle a_n\rangle
\]
We will also write
\[
n\cdot q:=\underset{n\text{ times}}{\underbrace{q\bot\ldots\bot q}}
\]

If $\mathrm{char}\,K=2$, then $q$ is diagonalizable if and only if
$q$ is totally singular. This is in contrast with the well known
case of $\mathrm{char}\,K\neq 2$, since over such fields every quadratic
form is diagonalizable. 

A quadratic form $q$ is called \emph{anisotropic} if $q(v)\neq 0$ for
every $0\neq v\in V$. Geometrically, this means that the associated
quadric $Q:=\{q=0\}\subset \P (V)$ does not have $K$-rational points.
We remark that the associated conic only depends on the similarity
class of the quadratic form.

A form $q$ is called \emph{non-degenerate} if it is regular and $\dim
r(b_q)\leq 1$. Geometrically speaking, non-degeneracy means that the
quadric $Q$ is smooth over $K$, while regularity means that the
quadric $Q$ is regular as a scheme, equivalently, is not a cone in
$\P(V)$ over a lower dimensional quadric.  In characteristic 2, there
can exist regular quadratic forms that fail to be non-degenerate.

For example, consider the subvariety $X$ of $\P^2_{(u:v:w)}\times
\P^2_{(x:y:z)}$ defined by $$ux^2 + vy^2 + wz^2 =0$$ over an
algebraically closed field $k$ of characteristic $2$. This is a conic
fibration over $\P^2_{(u:v:w)}$ such that the generic fiber $X_K$ over
$K=k(\P^2_{(u:v:w)})$ is defined by a quadratic form $q$ that is
anisotropic, but totally singular. The form $q$ is regular, but fails
to be non-degenerate. Geometrically, this means that the conic
fibration has no rational section (anisotropic), has a geometric
generic fibre that is a double line (totally singular), $X_K$ is not a
cone (regular), but $X_K$ is of course not smooth over~$K$. On the
other hand, the total space $X$ of this conic fibration is smooth
over~$k$.

One has the following structure theorem.

\begin{theorem}\label{structure_theorem_quadratic_forms_char_2}
Let $K$ be a field of characteristic $2$ and let $q$ be a quadratic
form on a finite-dimensional vector $V$ over $K$. Then there exist a
$m$-dimensional vector subspace $W\subseteq r(b_q)$ and
$2$-dimensional vector subspaces $V_1,\ldots,V_s\subseteq V$ such that
the following orthogonal decomposition is realized:
\[
q=q|_{r(q)}\bot q|_{W}\bot q|_{V_1}\bot\ldots\bot q|_{V_s}
\]
with $q|_{V_i}\simeq[a_i,b_i]$ for some $a_i,b_i\in K$
a non-degenerate form. Moreover, $q|_{W}$ is anisotropic, diagonalisable
and unique up to isometry. In particular,
\[
q\simeq r\cdot\langle 0\rangle\bot\langle c_1,\ldots,c_m\rangle\bot[a_1,b_1]\bot\ldots\bot[a_s,b_s]
\]
\end{theorem}

We now classify quadratic forms in three variables.







\begin{corollary}\label{cConics}
Let $K$ be a field of characteristic $2$, let $q$ be a nonzero
quadratic form in three variables over $K$, and let $Q \subset \P^2$
be the associated conic.  In the following table, we give the
classification of normal forms of $q$, up to similarity, and the
corresponding geometry of $Q$.
\[
\begin{array}{|c|c|c|c|}\hline
\dim r(q) & \dim r(b_q) &  \text{normal form of}~q &  \text{geometry
of}~Q \\\hline
0 & 1 & ax^2 + by^2 + xz + z^2 & \text{smooth conic} \\
0 & 3 & ax^2+by^2+z^2 & \text{regular conic, geom.\ double line}\\ 
1 & 1 & ax^2 + xz + z^2 & \text{cross of lines over $K_a$}\\
1 & 1 & xz & \text{cross of lines}\\
1 & 3 & ax^2+z^2 & \text{singular conic, geom.\ double line}\\
2 & 3 & z^2 & \text{double line}\\\hline
\end{array}
\]
Here, $(x:y:z)$ are homogeneous coordinates on $\P^2$; by cross of
lines, we mean a union of two disjoint lines in $\P^2$; and by $K_a$, we
mean the Artin--Schreier extension of~$K$ defined by $x^2-x-a$.
\end{corollary}

\begin{proof}
According to the classification in
Theorem~\ref{structure_theorem_quadratic_forms_char_2}, we have the
following normal forms for $q$ up to isometry over $K$.
\[
\begin{array}{|c|c|rl|}\hline
\dim r(q) & \dim r(b_q) &  \multicolumn{2}{c|}{\text{normal form of
$q$ up to isometry}} \\\hline
0 & 1 & ax^2 + by^2 + xz + cz^2 & \; a,c \in K,~b\in K\mult  \\
0 & 3 & ax^2+by^2+cz^2 & \; a,b,c \in K\mult \\ 
1 & 1 & ax^2 + xz + cz^2 & \; a,c \in K \\
1 & 3 & ax^2+cz^2 & \; a,c \in K\mult \\
2 & 3 & cz^2 & \; c \in K\mult \\\hline
\end{array}
\]
Here, in the cases $\dim r(q)\leq 1$ and $\dim r(b_q)=3$, we are
assuming that the associated diagonal quadratic forms $\langle a,b,c
\rangle$ in 3 variables or $\langle a,c \rangle$ in 2 variables,
respectively, are anisotropic.  Otherwise, these cases are not
necessarily distinct.

We remark that up to the change of variables $z \mapsto c\inv z$ and
multiplication by $c$, the quadratic forms $[a,c]$ and $[ac,1]$ are
similar.  Hence up to similarity, we can assume that $c=1$ in the
above table of normal forms up to isometry.

The fact that when $q$ is totally singular, $Q$ is geometrically a
double line follows since any diagonal quadratic form over an
algebraically closed form of characteristic 2 is the square of a
linear form.

Thus, the only case requiring attention is the case $\dim r(q)=\dim
r(b_q)=1$, where we claim that if $a = \alpha^2-\alpha$ for some
$\alpha \in K$, then $q$ is similar to $xz$ and thus $Q$ is a cross of
lines.  Indeed, after assuming that $c=1$, as above, we change
variables $z \mapsto z - \alpha x$ and $x \mapsto x-z$.  In
particular, when $a \in K/\wp(K)$ is nonzero, where $\wp : K \to K$ is
given by $\wp(x) = x^2-x$, then $Q$ becomes a cross of lines over the
Artin--Schreier extension $L/K$ defined by $x^2-x-a$.  This
distinguished the two cross of lines cases in the table in the
statement of the corollary.
\end{proof}


\subsection{Conic bundles.}

Let $k$ be an algebraically closed field. We adopt the following
definition of conic bundle.

\begin{definition} \label{dConicBundle} Let $X$ and $B$ be projective
varieties over $k$ and let $B$ be smooth. A \emph{conic bundle} is a
morphism $\pi:X\longrightarrow B$ such that $\pi$ is flat and proper
with every geometric fibre isomorphic to a plane conic and with smooth
geometric generic fibre. In practice, all conic bundles will be given
to us in the following form: there is a rank 3 vector bundle $\EE$
over $B$ and a quadratic form $q:\EE\longrightarrow\LL$ (with values
in some line bundle $\LL$ over $B$) which is not identically zero on
any fibre. Suppose that $q$ is non-degenerate on the generic fibre of
$\EE$. Then putting $X=\{q=0\}\subseteq\P(\EE)\longrightarrow B$,
where the arrow is the canonical projection map to $B$, defines a
conic bundle.
\end{definition}

The hypothesis on the geometric generic fibre is not redundant in our
context. Suppose that $\mathrm{char}\,k=2$, and let
$\pi:X\longrightarrow B$ be a flat, proper morphism such that every
geometric fibre is isomorphic to a plane conic. Let $\eta$ be the
generic point of $B$ and $K=k(B)$; note that the geometric generic
fibre $\ol{X}_{\eta}$ is a conic in $\P_{\ol{K}}^{2}$ and it is
defined by the vanishing of some quadratic form $q_{\eta}$.  By
Corollary~\ref{cConics}, we conclude that then $\ol{X}_{\eta}$ is cut
out by one of the following equations:
\begin{equation} \label{wild_cb}
ax^2+by^2+cz^2=0
\end{equation}
or
\begin{equation} \label{real_cb}
ax^2+by^2+yz+cz^2=0
\end{equation}
where $(x:y:z)$ are homogeneous coordinates for $\P_{\bar{K}}^{2}$.
The additional assumption on smoothness of the geometric generic fibre
allows us to rule out the case of \eqref{wild_cb}, which would give
rise to \emph{wild conic bundles}. 

We have to define discriminants of conic bundles together with their
scheme-structure.  First we discuss the discriminant of the generic conic.

\begin{remark}\label{rDiscriminants}
Let $\P^2$ have homogeneous coordinates $(x:y:z)$ and $\P =
\P(H^0(\P^2,\OO(2)))$ the 5-dimensional projective space of all conics
in $\P^2$.  We have the universal conic over $X_\univ \longrightarrow
\P$ defined as the projection of the incidence $X_\univ \subset \P
\times \P^2$, which can be written as a hypersurface of bidegree
$(1,2)$ defined by the generic conic
\[
a_{xx}x^2 + a_{yy} y^2 + a_{zz} z^2 + a_{xy}xy + a_{xz}xz + a_{yz} yz,
\]
where we can consider $(a_{xx}:a_{yy}:a_{zz}:a_{xy}:a_{xz}:a_{yz})$ as
a system of homogeneous coordinates on $\P$.  In these coordinates,
the equation of the discriminant $\Delta_\univ \subset \P$
parametrizing singular conics is
\begin{gather}\label{fDiscNot2}
4 a_{xx}a_{yy}a_{zz} + a_{xy}a_{yz}a_{xz} - a_{xz}^2 a_{yy} - a_{yz}^2 a_{xx} - a_{xy}^2 a_{zz}
\end{gather}
which simplifies to 
\begin{gather}\label{fDisc2}
a_{xy}a_{yz}a_{xz} + a_{xz}^2 a_{yy} + a_{yz}^2 a_{xx} + a_{xy}^2 a_{zz}
\end{gather}
in characteristic $2$. In any characteristic, $\Delta_\univ \subset
\P$ is a geometrically integral
hypersurface parameterizing the locus of singular conics in $\P^2$.
\end{remark}

\begin{definition}\label{dDiscriminants}
Let $\pi \colon X \to B$ be a conic bundle as in Definition \ref{dConicBundle}.
\begin{enumerate}
\item
The (geometric) discriminant $\Delta$ of the conic bundle is the union of those irreducible codimension $1$ subvarieties $\Delta_i$ in $B$ that have the following property: the geometric generic fibre of the restriction $\pi_{\pi^{-1}(\Delta_i)}\colon X_{\pi^{-1}(\Delta_i)} \to \Delta_i$ is not smooth.

\item We endow $\Delta$ with a scheme structure by assigning a
multiplicity to each $\Delta_i$ as follows. For each $i$, there is a
Zariski open dense subset $U_i\subset B$ such that $\Delta_i \cap U_i
\neq \emptyset$ and a morphism $f_i \colon U_i \to \P$ such that
$\pi|_{\pi^{-1}(U_i)} \colon X_{\pi^{-1}(U_i)} \to U_i$ is
isomorphic to the pull-back via $f_i$ of the universal conic bundle.
Then $\Delta_i \cap U_i$ is the reduced subscheme of a component of
$f_i^{-1}(\Delta_{\mathrm{univ}})$, interpreted as a scheme-theoretic
pullback, and we assign to $\Delta_i$ the corresponding multiplicity.
\end{enumerate}
\end{definition}

\section{Brauer groups and partially defined residues}\label{sBrauerMerkurjev}

For a Noetherian scheme $X$, we denote by $\Br(X)$ Grothendieck's
cohomological Brauer group, the torsion subgroup of
$\Het^2(X,\Gm)$. If $X = \mathrm{Spec}(A)$ for a commutative ring $A$,
we also write $\Br(A)$ for the Brauer group of $\mathrm{Spec}\, A$.

If $X$ is a regular scheme, every class in $\Het^2(X,\Gm)$ is torsion
\cite[II,~Prop.~1.4]{Gro68}. If $X$ is quasi-projective (over any
ring), a result of Gabber \cite{deJ03} says that this group equals the
Azumaya algebra Brauer group, defined as the group of Azumaya algebras
over $X$ up to Morita equivalence.

Below, unless mentioned otherwise, $X$ will be a smooth projective
variety over a field $k$.

In various applications, one is frequently given some highly singular
model of $X$ for which explicitly resolving is not feasible.  It is
thus desirable to be able to determine $\Br(X)$ purely in terms of
data associated with the function field $k(X)$. This is the idea
behind unramified invariants, e.g.\ \cite{Bogo87}, \cite{CTO},
\cite{CT95}. We have an inclusion
\begin{displaymath}\label{fInclusionBrauer}
\Br(X) \subset \Br(k(X))
\end{displaymath}
by \cite[II,~Cor.~1.10]{Gro68}, given by pulling back to the generic
point of $X$. One wants to single out the classes inside $\Br(k(X))$
that belong to $\Br(X)$ in valuation-theoretic terms. Since the basic
reference \cite{CT95} for this often only deals with the case of
torsion in the Brauer group coprime to the characteristic of $k$, we
gather together some results in the generality we will need.

Basic references for valuation theory are \cite{Z-S76} and
\cite{Vac06}. All valuations considered are Krull valuations.

\begin{definition}\label{dValBrauer}
Let $X$ be a smooth proper variety over a field $k$ and let
$\valus{S}$ be a subset of the set of all Krull valuations on the
function field $k(X)$. All valuations we consider will be
\emph{geometric}, i.e., they are assumed to be trivial on $k$. For
$v\in \valus{S}$, we denote by $A_v\subset k(X)$ the valuation ring of
$v$.  We denote by $\Br_{\valus{S}}(k(X)) \subset \Br(k(X))$ the set
of all Brauer classes $\alpha\in \Br(k(X))$ that are in the image of
the natural map $\Br(A_v) \to \Br(k(X))$ for all $v \in
\valus{S}$. Specifically, we consider the following sets $\valus{S}$.
\begin{enumerate}
\item The set $\valus{DISC}$ of all discrete rank $1$ valuations on
$k(X)$.

\item The set $\valus{DIV}$ of all divisorial valuations of
$k(X)$ corresponding to some prime divisor $D$ on a model $X'$ of
$k(X)$, where $X'$ is assumed to be generically smooth along $D$.

\item The set $\valus{DIV/}X$ of all divisorial valuations of $k(X)$
corresponding to a prime divisor on $X$.
\end{enumerate}
\end{definition}

Note the containments $\valus{DISC} \supset \valus{DIV} \supset
\valus{DIV/}X$, which are all strict in general.  Indeed, recall that
divisorial valuations are those discrete rank $1$ valuations $v$ with
the property that the transcendence degree of their residue field is
$\dim X-1$ \cite[Ch.~VI,~\S14,~p.~88]{Z-S76},
\cite[\S1.4,~Ex.~5]{Vac06}; and that there exist discrete rank $1$
valuations that are not divisorial, e.g., the analytic
arcs~\cite[Ex.~8(ii)]{Vac06}.  The main result that we need is the
following, which for torsion prime to the characteristic is proved in
\cite[Prop.~2.1.8,~\S2.2.2]{CT95}, and in general, in
\cite[Thm.~2.5]{ABBB18} using purity results due to
Gabber~\cite{Ga93}, \cite{Ga04}, \cite{Fuji02}, \cite{ILO14}, cf.\
\cite{Ces17}.

\begin{theorem}\label{tComparisonBrauer}
Let $X$ be a smooth projective variety over a field $k$.  Then all of
the natural inclusions
\[
\Br(X) \subset \Br_{\valus{DISC}}(k(X)) \subset \Br_{\valus{DIV}}(k(X)) \subset \Br_{\valus{DIV/}X}(k(X))
\]
are equalities.  In general, if $X$ is smooth and not necessarily
proper, then we still have the inclusion $\Br(X) \subset
\Br_{\valus{DIV/}X}(k(X))$ and this is an equality.
\end{theorem}

\begin{remark}\label{rUnramifiedBrauer}
In the setting of Theorem~\ref{tComparisonBrauer}, we will agree to
denote the group $\Br_{\valus{DIV}}(k(X))$ by $\Br_\ur(k(X))$ and
call this the unramified Brauer group of the function field $k(X)$. We
will also use this notation for singular $X$. According to
\cite{Hi17}, a resolution of singularities should always exist, but we
do not need this result: in all our applications we will produce explicit
resolutions $\widetilde{X}$, and then we know
$\Br_\ur(k(X))=\Br(\widetilde{X})$.
\end{remark}

Next we want to characterize elements in $\Br_\ur(k(X))$ in terms of
partially defined residues in the sense of Merkurjev as in
\cite[Appendix A]{GMS03}; this is necessitated by the following
circumstance: if one wants to give a formula for the unramified Brauer
group of a conic bundle over some smooth projective rational base
(see, e.g., \cite{Pi16}, \cite{ABBP16}), for example a smooth
projective rational surface, the idea of \cite{CTO} is to produce the
nonzero Brauer classes on the total space of the given conic bundle as
pull-backs of Brauer classes represented by certain other conic
bundles on the base whose residue profiles are a proper subset of the
residue profile of the given conic bundle. Hence, one also has to
understand the geometric meaning of residues because in the course of
this approach it becomes necessary to decide when the residues of two
conic bundles along one and the same divisor are equal.



Let $K$ be a field of characteristic $p$. We denote the subgroup of
elements in $\Br(K)$ whose order equals a power of $p$ by $\Br(K)\{ p
\}$.  Let $v$ be a discrete valuation of $K$, $K_v$ the completion of
$K$ with respect to the absolute value induced by $v$. Denote the
residue field of $v$ by $k(v)$ and by $\ol{K}_v$ an algebraic closure
of $K_v$. One can extend $v$ uniquely from $K_v$ to $\ol{K}_v$, the
residue field for that extended valuation on $\ol{K}_v$ will be
denoted by $\overline{k(v)}$.

By \cite[p.~64--67]{Artin67}, unramified subfields of $\ol{K}_v$
correspond to separable subfields of $\overline{k(v)}$, and, in
particular, there is a maximal unramified extension, with residue
field $k(v)\sep$ (separable closure), called the
\emph{inertia field}, and denoted by $K_v^\ur$ or $T=T_v$ (for
\emph{Tr\"agheitsk\"orper}). One also has that the Galois group
$\Gal(K_v^\ur/K_v)$ is isomorphic to $\Gal(k(v))$.
Now, recall that by the Galois cohomology characterization, the Brauer
group $\Br(K_v)\{ p\}$ is isomorphic to $H^2 (K_v,
K_v\sep{}\mult)\{ p\}$ and there is a natural map
\begin{equation} \label{fGaloisNatMap}
H^2 (\Gal(K_v^\ur/K_v), K_v^\ur{}\mult) \{ p \} \to H^2 (K_v, K_v\sep{}\mult)\{ p\}
\end{equation}
which is injective \cite[App.~A,~Lemma~A.6,~p.~153]{GMS03}.

\begin{definition}\label{defTameSubgroup}
With the above setting, we call the image of (\ref{fGaloisNatMap}) the \emph{tame subgroup} or \emph{tamely ramified} subgroup of $\Br(K_v)\{ p\}$ associated to $v$, and denote it by $\Br_{\mathrm{tame}, v}(K_v)\{p\}$. We denote its preimage in $\Br(K)\{ p\}$ by $\Br_{\mathrm{tame}, v}(K)\{p\}$ and likewise call it the tame subgroup of $\Br(K)\{ p\}$ associated to $v$.
\end{definition}

Writing again $v$ for the unique extension of $v$ to $K_v^\ur$ we have a group homomorphism
\[
v\colon K_v^\ur{}\mult \to \Z
\]
\begin{definition} \label{dMerkurjevResidues}
Following \cite[App. A]{GMS03} one can define a map as the composition
\small
\[
r_v \colon \Br_{\mathrm{tame}, v}(K)\{p\} \to H^2 (\Gal(K_v^\ur/K_v), K_v^\ur{}\mult)\{ p\}  \to H^2 (k(v), \Z )\{p\}\simeq H^1 (k(v), \Q /\Z)\{p\}. 
\]
\normalsize which we call the \emph{residue map} with respect to the
valuation $v$. We will say that the residue of an element $\alpha \in
\Br(K)\{ p\}$ with respect to a valuation $v$ is \emph{defined},
equivalently, that $\alpha$ is \emph{tamely ramified} at $v$, if
$\alpha$ is contained in $\Br_{\mathrm{tame}, v}(K)\{p\}$.
\end{definition}

\begin{remark}\label{rArtinSchreier}
If $\alpha \in \Br(K)[p]$ for which the residue with respect to a
valuation $v$ is defined, as in Definition~\ref{dMerkurjevResidues},
then $r_v (\alpha ) \in H^1 (k(v), \Z /p)$. By Artin--Schreier theory
\cite[Prop.~4.3.10]{GS06}, one has $H^1 (k(v), \Z /p)\simeq k(v)/\wp
(k(v))$ where $\wp \colon k(v) \to k(v)$, $\wp (x) =x^p -x$, is the
Artin--Schreier map. This group classifies pairs, consisting of a
finite $\Z/p$-Galois extension of $k(v)$ together with a chosen
generator of the Galois group.  Indeed, $\Z/p$-Galois extensions of
$k(v)$ are Artin--Schreier extensions, i.e., generated by the roots of
a polynomial $x^p- x-a$ for some $a\in k(v)$.  The isomorphism class
of such an Artin--Schreier extension is unique up to the substitution
\[
a\mapsto \eta a + (c^p -c)
\]
where $\eta\in \F_p^{\times}$ and $c\in k(v)$, see for example
\cite[\S7.2]{Artin07}. In particular, for $p=2$, one may also identify
$H^1 (k(v), \Z /2)$ with $\mathrm{\acute{E}t}_2 (k(v))$, the set of
isomorphism classes \'{e}tale algebras of degree $2$ over $k(v)$, cf.\
\cite[p.~402,~Ex.~101.1]{EKM08}; more geometrically, if $D$ is a prime
divisor on a smooth algebraic variety over a field $k$ and $v_D$ the
corresponding valuation, the residue can be thought of as being given
by an \'{e}tale double cover of an open part of $D$.
\end{remark}

\begin{remark}\label{rRelativeBrauer}
Keep the notation of Definition \ref{dMerkurjevResidues}. The tame
subgroup
\[
\Br_{\mathrm{tame}, v}(K_v)\{p\}=H^2 (\Gal(K_v^\ur/K_v),
K_v^\ur{}\mult) \{ p \}
\] 
of $\Br(K_v)\{ p\}$ has a simpler description by
\cite[Thm.~4.4.7,~Def.~2.4.9]{GS06}: it is nothing but the subgroup of
elements of order a power of $p$ in the relative Brauer group
$\Br(K_v^\ur/K_v)$ of Brauer classes in $\Br(K_v)$ that are split by
the inertia field $K_v^\ur$,
in other words, are in the kernel of the natural map
\[
\Br(K_v) \to \Br(K_v^\ur).
\]
To explain the name, one can say that the tame subgroup of $\Br(K_v)\{
p\}$ consists of those classes that become trivial in $\Br(V)$, where
$V$ is the \emph{maximal tamely ramified extension} of $K_v$, the
ramification field (\emph{Verzweigungsk\"orper})
\cite[Ch.~4,~\S2]{Artin67}, because the classes of orders a power of
$p$ split by $ \Br(T)\{p\}$ coincide with those split by $\Br(V)\{p\}$
(since $V$ is obtained from $T$ by adjoining roots $\sqrt[m]{\pi}$ of
a uniformizing element $\pi$ of $K_v$ of orders $m$ not divisible by
$p$ and restriction followed by corestriction is multiplication by the
degree of a finite extension). Since in characteristic~$0$ every
extension of $K_v$ is tamely ramified, one can say that in general
residues are defined on the subgroup of those classes in
$\Br(K_v)\{p\}$ that become trivial on $\Br(V)\{p\}$.

The terminology tame subgroup was suggested to us by Burt Totaro, who
also kindly provided other references to the literature. It has the
advantage of avoiding the confusing terminology ``unramified
subgroup", also sometimes used, for elements that can have nontrivial
residues. The terminology here is consistent with
\cite[\S6.2,~Prop.~6.63]{TiWa15}, and one could also have called the
tamely ramified subgroup the \emph{inertially split part}, following
that source, as the two notions coincide in our context. Our
terminology is also consistent with the one in \cite[Thm.~3]{Ka82}.
\end{remark}

\begin{remark}\label{rTameSubgroupsForHigherCohomologies}
More generally, given a field $F$ of characteristic $p>0$, one can
define a version of Galois cohomology ``with mod $p$ coefficients",
following Kato \cite{Ka86} or Merkurjev \cite[App.~A]{GMS03},
\cite{Mer15}, in the following way: define
\begin{equation}\label{fHigherCohomology}
H^{n+1}(F,\Q_p/\Z_p(n)):=H^2 (F, \mathrm{K}_n(F\sep ))\{ p\}
\end{equation}
where $\mathrm{K}_{n}(F\sep)$ is the $n$-th Milnor K-group of the
separable closure of $F$, and the cohomology on the right hand side is
usual Galois cohomology with coefficients in this Galois module. The
coefficients $\Q_p/\Z_p(n)$ on the left hand side are just a symbol
here to point out the similarity with the case of characteristics
coprime to $p$, though one can also define them via the logarithmic
part of the de Rham--Witt complex, where this symbol has meaning as a
coefficients complex.

Given a discrete rank $1$ valuation $v$ of $F$ with residue field $E$, one can define a \emph{tame subgroup} (or \emph{tamely ramified subgroup}) 
\[
H^{n+1}_{\mathrm{tame}, v} (F, \Q_p /\Z_p (n)) \subset H^{n+1} (F, \Q_p /\Z_p (n))
\]
in this more general setting in such a way that one recovers the
definition given for the Brauer group in
Definition~\ref{dMerkurjevResidues} above: following
\cite[p.~153]{GMS03} let $F_v$ be the completion, $F_v^\ur$ its
maximal unramified extension, and put
\[
H^{n+1}_{\mathrm{tame}, v}(F_v, \Q_p /\Z_p (n)) := H^2 (\Gal(F_v^\ur/F_v), \mathrm{K}_n (F_v^\ur))\{ p\} \subset H^2 (F_v, \mathrm{K}_n (F_v^{\mathrm{sep}}))\{ p\}
\]
(this is actually a subgroup by \cite[Lemma A.6]{GMS03}). Then define
the subgroup $H^{n+1}_{\mathrm{tame}, v} (F, \Q_p/\Z_p (n))$ as the
preimage of $H^{n+1}_{\mathrm{tame}, v}(F_v, \Q_p/\Z_p (n))$ under the
natural map $H^{n+1} (F, \Q_p/\Z_p (n)) \to H^{n+1} (F_v, \Q_p/\Z_p
(n))$. The $\Gal(F_v^\ur/F_v)$-equivariant residue map in
Milnor K-theory
\[
\mathrm{K}_n (F_v^\ur) \to K_{n-1}(E\sep)
\]
then induces a residue map, defined only on $H^{n+1}_{\mathrm{tame}, v} (F, \Q_p/\Z_p (n))$, 
\[
r_v \colon H^{n+1}_{\mathrm{tame}, v} (F, \Q_p/\Z_p (n)) \to H^n (E, \Q_p/\Z_p (n-1)).
\]
Note that, naturally, $\Gal(F_v^\ur/F_v) \simeq \Gal(E)$. 

\medskip

We want to describe the relation to logarithmic differentials and restrict to the case of $p$-torsion for simplicity. Given a discrete rank $1$ valuation $v$ of $F$ with residue field $E$, we have the group 
\[
H^{n+1}_{\mathrm{tame}, v} (F, \Z /p (n))=H^{n+1}_{\mathrm{tame},v}(F,\Q_p/\Z_p(n))[p]
\]
and there is a residue map $r_v$ defined on $H^{n+1}_{\mathrm{tame},
v} (F, \Z /p (n))$ as the restriction of the above $r_v$. We now have
the following alternative description
\[
H^{n+1} (F, \Z /p (n)) = H^1 (F, \Omega^n_{\mathrm{log}}(F\sep))
\]
where the coefficients $\Omega^n_{\mathrm{log}}(F\sep)$, denoted $\nu (n)_{F\sep}$ in other sources, are defined as the kernel in the exact sequence of Galois modules 
\[
\xymatrix{
0 \ar[r] & \Omega^n_{\mathrm{log}}(F\sep) \ar[r] & \Omega^n_{F\sep} \ar[r]^{\gamma -1 \quad\quad} & \Omega^n_{F\sep} / B^n_{F\sep} \ar[r] & 0
}
\]
see \cite[after~Prop.~1.4.2]{CT99}; here $B^n_{F\sep}$ is the subspace
of boundaries, the image of the differential $d\colon
\Omega^{n-1}_{F\sep} \to \Omega^n_{F\sep}$, and $\gamma - 1$ is a
generalization of the Artin--Schreier map defined on generators as
\begin{align*}
\gamma - 1 \colon  \Omega^n_{F\sep} & \to \Omega^n_{F\sep} / B^n_{F\sep}\\
              x \frac{dy_1}{y_1} \wedge \dots \wedge \frac{dy_n}{y_n}  & \mapsto (x^p -x) \frac{dy_1}{y_1} \wedge \dots \wedge \frac{dy_n}{y_n}\quad \mathrm{mod}\; B^n_{F\sep}
\end{align*}
with $x\in F\sep$, $y_i \in F\sep{}\mult$. Now by a result of Kato,
Bloch--Kato, Gabber, cf.\ \cite[Thm.~3.0]{CT99}, one has
\[
 \Omega^n_{\mathrm{log}}(F\sep) \simeq \mathrm{K}_n (F\sep)/p
\]
or, since by Izhboldin's theorem \cite[Thm.~7.8,~p.~274]{Wei13} the
groups $\mathrm{K}_n (F\sep)$ have no $p$-torsion, an exact sequence
\[
\xymatrix{
1 \ar[r] & \mathrm{K}_n (F\sep) \ar[r]^{\times p} & \mathrm{K}_n (F\sep) \ar[r] & \Omega^n_{\mathrm{log}}(F\sep) \ar[r] & 1
}
\]
of $\Gal(F)$-modules. The latter shows the equivalence of our
two definitions by passing to the long exact sequence in Galois
cohomology, taking into account that
\[
H^1 (F, \mathrm{K}_n (F\sep)) =0
\]
by \cite[Lemma 6.6]{Izh91}. One then finds 
\begin{align*}
H^1 (F, \Z /p (0)) &\simeq H^1 (F, \Z /p) \simeq F/\wp (F)\\
H^2 (F, \Z /p (1)) & \simeq \Br(F)[p]
\end{align*}
which brings us back again to Definition \ref{dMerkurjevResidues}.
See also \cite[p.~152,~Ex.~A.3]{GMS03}.
\end{remark}

We can use the two preceding remarks to prove the following.

\begin{theorem}\label{tResConicBundleAlongDiv}
Let $k$ be an algebraically closed field of characteristic $2$. Let $X$ and $B$ be projective varieties over $k$, let $B$ be smooth of dimension $\ge 2$ and let $\pi \colon X \to B$ be a conic bundle. Let $K$ be the function field of $B$, let $\alpha \in \Br(K)[2]$ be the Brauer class determined by the conic bundle, and let $D$ be a prime divisor on $B$. Suppose that one is in either one of the following two cases:
\begin{enumerate}
\item
The geometric generic fibre of 
\[
\pi|_{\pi^{-1}(D)} \colon X|_{\pi^{-1}(D)} \to D
\]
is a smooth conic.
\item
The geometric generic fibre of 
\[
\pi|_{\pi^{-1}(D)} \colon X|_{\pi^{-1}(D)} \to D
\]
is isomorphic to two distinct lines in the projective plane and $D$ is
a reduced component of the discriminant $\Delta$ as in Definition
\ref{dDiscriminants}.  In this case, we say that the conic bundle is
\emph{tamely ramified} over $D$.
\end{enumerate}
Then in both cases, the residue of $\alpha$ with respect to the
divisorial valuation $v_D$ determined by $D$ is defined, and in case
\textit{a)} it is zero, whereas in case \textit{b)} it is the class of the double cover
of $D$ induced by the restriction of the conic bundle $\pi$ over $D$,
which is \'{e}tale over an open part of $D$ by the assumption on the
type of geometric generic fibre.
\end{theorem}

We need the following auxiliary results before commencing with the proof.

\begin{lemma}\label{lNormalFormHC}
Under all the hypotheses of \textit{a)} or \textit{b)} of Theorem \ref{tResConicBundleAlongDiv}, except possibly the reducedness of $D$, let $P\in D$ be a point where the fibre $X_P$ is a smooth conic or a cross of lines, respectively. Then we can assume that Zariski locally around $P$ the conic bundle is defined by
\[
a x^2 + b y^2 + xz + z^2 =0
\]
with $x,y,z$ fibre coordinates and $a, b$ functions on $B$, both regular locally around $P$ and $b$ not identically zero. 
\end{lemma}

\begin{proof}
Locally around $P$, the conic bundle is given by an equation
\[
a_{xx} x^2 + a_{yy}y^2 + a_{zz}z^2 + a_{xy} xy + a_{xz} xz + a_{yz} yz =0
\]
with the $a$'s regular functions locally around $P$. 
Since the fibre $X_P$ is smooth or a cross of lines, we have that one of the coefficients of the mixed terms, without loss of generality $a_{xz}$, is nonzero in $P$. Introducing new coordinates by the substitution $x \mapsto (1/a_{xz}) x$ (here and in the following we treat $x,y,z$ as well as the $a$'s as dynamical variables to ease notation) one gets the form
\[
a_{xx} x^2 + a_{yy}y^2 + a_{zz}z^2 + a_{xy} xy + xz + a_{yz} yz =0.
\]
Now the substitutions $x \mapsto x+ a_{yz} y, \quad y\mapsto y, \quad z \mapsto z + a_{xy} y$ transforms this into
\[
a_{xx} x^2 + a_{yy} y^2 + a_{zz}z^2 + xz = 0.
\]
Now if one of $a_{xx}$ or $a_{zz}$ is nonzero in $P$, without loss of generality $a_{zz}(P)\neq 0$, then multiplying the equation by $a_{zz}^{-1}$ and subsequently applying the substitution $x\mapsto a_{zz} x$ we obtain the desired normal form
\begin{gather}\label{fNormal}
a_{xx}x^2 + a_{yy} y^2 + xz + z^2 =0.
\end{gather}
But if both $a_{xx}$ and $a_{zz}$ vanish in $P$, then after applying
the substitution $x\mapsto x+z$, we get that $a_{zz}(P) \neq 0$ and
proceed as before.
\end{proof}

Before stating the next auxiliary result, we recall the existence of a
cup product homomorphism
\[
K_i(K) \tensor H^{n+1}(K,\Q_p/\Z_p(n)) \to H^{n+i+1}(K,\Q_p/\Z_p(n+i)),
\]
which restricts to the tamely ramified subgroups, see
\cite[p.~154,~(A.7)]{GMS03}.

\begin{proposition}\label{pCupProduct}
Let $K$ be a field of characteristic $2$ and $Q$ the conic defined by
$ax^2 + by^2 + xz + z^2=0$ for $a \in K$ and $b \in K\mult$.  Then the
Brauer class associated to $Q$ is the cup product $b \cup a$, via the
cup product homomorphism
$$
K_1(K) \tensor H^1(K,\Q_2/\Z_2(0)) \to H^2(K,\Q_2/\Z_2(1)) = \Br(K)\{2\},
$$
where we consider $a \in H^1(K,\Q_2/\Z_2)[2] = K/\wp(K)$ and $b \in
K_1(K) = K\mult$.
\end{proposition}

\begin{proof}
The Brauer class associated to $Q$ is the generalized quaternion
algebra $[a,b)$, defined as the free associative $K$-algebra on
generators $i$ and $j$ with the relations
\[
i^2-i=a, \quad j^2=b, \quad ij = ji + j,
\] 
see \cite[Ch.~1,~Exer.~4]{GS06}.  Let $L/K$ be the Artin--Schreier
extension, which could be the split \'etale algebra, generated by
$x^2-x-a$ and $\chi_{L/K} : \Gal(K\sep/K) \to \Z/2\Z$ the canonically
associated character of the absolute Galois group of $K$.  
By
\cite[Cor.~2.5.5b]{GS06}, the quaternion algebra $[a,b)$ is
$K$-isomorphic to the cyclic algebra $(\chi_{L/K},b)$, generated as a
$K$-algebra by $L$ and an element $y$ with the relations
\[
y^2=b, \quad \lambda y = y \sigma(\lambda)
\]
where $\lambda \in L$ and $\sigma$ is the generator of $\Gal(L/K)$.

Letting $\delta : H^1(K,\Z/2) \to H^2(K,\Z)$ be the \emph{Bockstein
homomorphism} induced from the coboundary map in Galois cohomology
associated to the exact sequence of trivial Galois modules
\[
0 \to \Z \to \Z \to \Z/2 \to 0
\]
then $\delta : H^1(K,\Z/2) \to H^2(K,\Z)[2]$ is an isomorphism that
gives sense to Merkurjev's definition $H^1(K,Z/2Z(0)) :=
H^2(K,K_0(F\sep))[2] = H^2(K,\Z)[2]$.  By \cite[Prop.~4.7.3]{GS06},
the cup product pairing in Galois cohomology
\[
H^2(K,\Z) \times H^0(K,K\sep{}\mult) \to H^2(K,K\sep{}\mult) = \Br(F)
\]
has the property that the cup product of the class $\delta(a) \in
H^2(K,\Z)[2]$, where we consider $a\in K/\wp(K)=H^1(K,\Z/2)$, with the
class $b \in H^0(K,K\sep{}\mult) = K\mult$, results in the Brauer class of
the cyclic algebra $(\chi_{L/K},b) \in \Br(K)[2]$.

Finally, under the canonical identification $H^0(K,K\sep{}\mult) =
K\mult = K_1(K)$, the isomorphism $\delta : H^1(K,\Z/2) \to
H^2(K,\Z)[2]$, and the definition of the action of $K_1(K)$ on
$H^1(K,\Q_2/\Z_2(0))$, the cup product pairing
\[
H^1(K,\Z/2(0)) \times K_1(K) \to H^2(K,\Z/2(1)) = \Br(K)[2]
\]
is identified with the $2$-torsion part of the above cup product
pairing in Galois cohomology.  Since the cup product is commutative on
$2$-torsion classes, we get the desired formula.

We also point out that the relevant cocycle calculation in the proof
of this result in \cite[Prop.~30.4]{KMRT98}, though stated for $K$ of
characteristic not $2$, can be generalized to the case of
characteristic $2$ using the machinery of flat cohomology.
\end{proof}

\begin{proof}[Proof of Theorem \ref{tResConicBundleAlongDiv}]
Let $v=v_D$ be the divisorial valuation associated to $D$. We have to
check that in both cases of the Theorem, $\alpha\in \Br(K_v^\ur/K_v)$,
see Remark \ref{rRelativeBrauer}, in other words, that $\alpha$ is
split by $K_v^\ur$. We have the normal form of Lemma
\ref{lNormalFormHC} locally around the generic point of $D$. The conic
bundle is then obviously split by the Galois cover of the base defined
by adjoining to $k(B)$ the roots of $T^2 + T +a$ because then the
quadratic form in Lemma \ref{lNormalFormHC} acquires a zero. Moreover,
that Galois cover does not ramify in the generic point of $D$ because
$a$ has no pole along $D$.  Hence it defines an extension of $K_v$
contained in $K^\ur_v$. See also \cite[Ex.~1,~p.~67]{Artin67}.

By formula (\ref{fDisc2}) we find that the discriminant of $ax^2 +
by^2 + xz + z^2=0$ is given by $b$, hence by our assumption that $D$
is reduced, we can assume $b$ is a local parameter for $D$, in case
\textit{b)} of the Theorem, or a unit in the generic point of $D$ in
case~\textit{a)}. By Proposition~\ref{pCupProduct}, the Brauer class
$\alpha \in \Br(k(B))\{ 2\}= H^2 (k(B), \Q_2/\Z_2 (1))$ associated to
the conic bundle defined by the preceding formula is the cup product
$\alpha = b \cup a$ of the class $b \in K_1 (k(B)) = k(B)\mult$ and
the class $a \in H^1 (k(B), \Q_2/\Z_2)=k(B)/\wp(k(B))$.  We now
conclude the proof in a number of steps.

\textbf{Step 1.} If $\pi$ is a local equation for $D$ in $\OO_{B, D}$, then a polynomial in $a$ with coefficients in $k$ can only vanish along $D$ if $a$ is congruent modulo $\pi$ to some element in $k\mult$. If $a$ is not congruent modulo $\pi$ to an element in $k$, we consequently have that $k(a) \subset k(B)$ is a subfield of the valuation ring $\OO_{B, D}$ of $v$. By \cite[p.~154, sentence before formula (A.8)]{GMS03}, the element that $a$ induces in $H^1 (k(B), \Q_2/\Z_2)$ is in $H^1_{\mathrm{tame}, v} (k(B), \Q_2/\Z_2)$, and then formula (A.8) of loc.cit.\ implies
\[
r_v (b \cup a) = a|_D \in k(D)/\wp (k(D)) = H^1 (k(D), \Z /2)
\]
in case \textit{b)} of the Theorem, and $r_v(b\cup a)=0$ in case \textit{a)} because the element $b$ is then a unit in the valuation ring of $D$ (alternatively, in case \textit{a)}, the Brauer class of the conic bundle clearly comes from $\Br(A)$, where $A$ is the valuation ring of $D$, hence the residue is defined and is zero; see also proof of Theorem \ref{tResMerkUnr} below).  Since, in case \textit{b)}, $a|_D$ is precisely the element defining the Artin--Schreier double cover induced by the conic bundle on $D=(b=0)$, the residue is given by this geometrically defined double cover.

\textbf{Step 2.}  If $a$ is congruent modulo $\pi$ to an element in
$k$, and since $k$ is algebraically closed, we can make a change in
the fibre coordinate $z$ so that $a$ is actually a power of $\pi$
times a unit in $\OO_{B, D}$. Since $\dim B \ge 2$, we can find a unit
$a'\in \OO_{B, D}$ that is not congruent to an element in $k$ modulo
$\pi$, and write $a = (a-a')+ a'$. Now applying Step 1 to $a-a'$ and
$a'$ finishes the proof since the cup product $\cup$ is bilinear and
$r_v$ is linear, so $a|_D$, the element defining the Artin--Schreier
double cover induced by the conic bundle, is equal to the residue of
the conic bundle along $D$ in general.
\end{proof}

\begin{lemma}\label{BrauerGrDvrAlessandro}
Let $R$ be a complete discrete valuation ring with field of fractions $K$ and let $K^\mathrm{nr}$ be the maximal unramified extension of $K$, as before. Let $R^\mathrm{nr}$ be the integral closure of $R$ in $K^\mathrm{nr}$. Then $\Br(R)=H^2 (\Gal(K^\ur/K),R^\ur{}\mult)$.
\end{lemma}
\begin{proof}
This is contained in\cite{AB68}, see the proof at the top of page 289, combined
with the remark in \S3, and the first sentence of the proof of Theorem~3.1.
\end{proof}

\begin{theorem}\label{tResMerkUnr}
Let $X$ be a smooth and projective variety over an algebraically
closed field $k$ of characteristic $p$. Assume $\alpha\in \Br(k(X))\{
p\}$ is such that the residue $r_{v_D}(\alpha )$ is defined in the
sense of Definition \ref{dMerkurjevResidues} and is trivial for all
divisorial valuations $v_D$ corresponding to prime divisors $D$ on
$X$. Then $\alpha \in \Br_\ur(k(X)) = \Br(X)$.

If $Z \subset X$ is an irreducible subvariety with local ring $\OO_{X, Z}$ and the assumption above is only required to hold for all prime divisors $D$ passing through $Z$, the class $\alpha$ comes from $\Br(\OO_{X, Z})$.
\end{theorem}

\begin{proof}
The equality $\Br_\ur(k(X)) = \Br(X)$ follows from Theorem
\ref{tComparisonBrauer} taking into account
Remark~\ref{rUnramifiedBrauer}. We will show that under the
assumptions above, we have $\alpha \in
\Br_{\valus{DIV/}X}(k(X))$, which is enough by Theorem
\ref{tComparisonBrauer}. Putting $K=k(X)$ and $v=v_D$ and keeping the
notation of Definition \ref{dMerkurjevResidues} we have an exact
sequence\small
\[
\xymatrix{
H^2 (\Gal(K_v^\ur/K_v), A_v^\ur{}\mult)\{p\} \ar[r]   & H^2 (\Gal(K_v^\ur/K_v), K_v^\ur{}\mult)\{ p\} \ar[r]^(.575){r_v} & H^1 (k(v), \Q /\Z)\{p\} 
}
\]
\normalsize
resulting from the exact sequence of coefficients $ 1 \to A_v^\ur{}\mult \to K_v^\ur{}\mult \to \Z \to 1$ where $A_v^\ur{}\mult$ is the valuation ring, inside of $K_v^\ur$, of the extension of $v$ to $K_v^\ur$. Thus it suffices to show that classes in $\Br_{\mathrm{tame}, v}(K)\{p\} \subset \Br(K)\{p\}$ that, under the map
\[
\Br_{\mathrm{tame}, v}(K)\{ p\} \to H^2 (\Gal(K_v^\ur/K_v), K_v^\ur{}\mult)\{ p\} \subset \Br(K_v)\{p\},
\]
land in the image of $H^2 (\Gal(K_v^\ur/K_v), A_v^\ur{}\mult)\{p\}$ actually come from $\Br(\OO_{X, \xi_D})\{p\}$ where $\OO_{X, \xi_D}$ is the local ring of $D$ in $k(X)$. Now, by Lemma \ref{BrauerGrDvrAlessandro}, we have
\[
H^2 (\Gal(K_v^\ur/K_v), A_v^\ur{}\mult) \simeq \Br(A_v)
\]
A class $\gamma$ in $\Br(K)$ whose image $\gamma_v$ in $\Br(K_v)$ is contained in $\Br(A_v)$ comes from the valuation ring $A= \OO_{X, \xi_D}$ of $v$ in $K$ by Lemma \ref{lAsher} below, hence is unramified. 
\end{proof}

\begin{lemma}\label{lAsher}
Let $K$ be the function field of an algebraic variety and $v$ a discrete rank $1$ valuation of $K$. Let $A\subset K$ be the valuation ring, let $K_v$ be the completion of $K$ with respect to $v$, and let $A_v \subset K_v$ be the valuation ring of the extension of $v$ to $K_v$. Then a Brauer class $\alpha \in \Br(K)$ whose image in $\Br(K_v)$ comes from a class $\alpha^{\sharp} \in \Br(A_v)$ is already in the image of $\Br(A)$. 
\end{lemma}

\begin{proof}
This is a special case of \cite[Lemma 4.1.3]{Ha67} or \cite[Lemma 4.1]{CTPS12}, but we include a proof for completeness. 

Suppose the class $\alpha$ is represented by an Azumaya algebra $\sAA$
over $K$, and that $\alpha^{\sharp}$ is represented by an Azumaya
algebra $\BB$ over $A_v$. By assumption, $\sAA$ and $\BB$ become
Brauer equivalent over $K_v$, and we can assume that they even become
isomorphic over $K_v$ by replacing $\sAA$ and $\BB$ by matrix algebras
over them so that they have the same degree. Let $\sAA_A$ be a maximal
$A$-order of the algebra $\sAA$ in the sense of Auslander--Goldman
\cite{AG60}, which means that $\sAA_A$ is a subring of $\sAA$ that is
finitely generated as an $A$-module, spans $\sAA$ over $K$ and is
maximal with these properties. We seek to prove that $\sAA_A$ is
Azumaya.  Now we know that the base change $(\sAA_A)_{A_v}$ is a
maximal order, but also any Azumaya $A_v$-algebra is a maximal order,
and by \cite[Prop.~3.5]{AG60}, any two maximal orders over a rank $1$
discrete valuation ring are conjugate, so in fact the base change
$(\sAA_A)_{A_v}$ is Azumaya because $\BB$ is.  But then this implies
that $\sAA_A$ is Azumaya since $A_v$ is faithfully flat over $A$, so
if the canonical algebra homomorphism $\sAA_A \otimes \sAA_A \to
\mathrm{End} (\sAA_A)$ becomes an isomorphism over $A_v$, it is
already an isomorphism over $A$.
\end{proof}

\begin{remark}\label{rMoreGeometricProofHC}
Here is a more geometric proof of Theorem \ref{tResMerkUnr} for the case that $D$ is as in Theorem \ref{tResConicBundleAlongDiv}\textit{b)} and the class $\alpha$ is a $2$-torsion class represented by a conic bundle $\pi \colon Y \to X$ itself. In that case, it suffices to show that there exists a birational modification 
\[
\xymatrix{
Y' \ar[d]\ar@{-->}[r] & Y \ar[d] \\
X' \ar[r]  & X
}
\]
with $Y' \to X'$ a conic bundle square birational to $Y\to X$, $X'\to X$ an isomorphism over the generic point of $D$ and such that the general fibre of $Y'$ over the strict transform of $D$ on $X'$ is a smooth conic. 

We can assume the normal form from Lemma \ref{lNormalFormHC}
\[
a x^2 + b y^2 = xz + z^2
\]
whence in characteristic $2$
\[
b y^2 = ax^2 + xz + z^2.
\]
By assumption, the right hand side factors modulo $b$, that is, there is a function $\alpha$ on $(b=0)$ such that
\[
(\alpha x + z) (1+\alpha x + z) = a|_{(b=0)} x^2 + xz + z^2. 
\]
Let $\alpha'$ some extension of $\alpha$ to a neighborhood of $(b=0)$ in the base. Then the matrix
\[
\begin{pmatrix}
\alpha'  & 1 \\
1+\alpha' &       1
\end{pmatrix}
\]
has determinant $1$, and thus applying the coordinate transformation 
\begin{align*}
x' &= \alpha' x + z\\
z' &= (1+ \alpha')x + z
\end{align*}
one gets 
\[
b y^2 = x'z' + b(u (x')^2 + v x'z' + w (z')^2).
\]
Applying $x'\mapsto b x''$ one gets
\[
b y^2 = b x''z' + b(ub^2 (x'')^2 + vb x''z' + w (z')^2)
\]
Outside of $b=0$ one can divide by $b$ and gets
\[
y^2 = x''z' + ub^2 (x'')^2 + vb x''z' + w (z')^2.
\]
The derivative with respect to $x''$ is
\[
z' + vb z' = (1+vb)z'
\]
and the derivative with respect to $z'$
\[
(1+vb)x''.
\]
Now $1+vb$ is invertible in a neighborhood of the generic point of $D$ and thus the singularities of this conic bundle are contained in
\[
x''=z'=0.
\]
Substituting in the given equation we also get $y^2=0$ and thus the transformed conic bundle has smooth total space over a neighborhood of the generic point of $D$. 
\end{remark}

\section{Discriminant profiles of conic bundles in characteristic two: an instructive example}\label{sDiscriminantsChar2}

In this Section, we work over an algebraically closed ground field $k$. First $k$ may have arbitrary characteristic, later we will focus on the characteristic two case. Let $X$ be a conic bundle over a smooth projective base $B$ as in Definition \ref{dConicBundle}.

\begin{definition}\label{dResidueProfile}
We denote by $B^{(1)}$ the set of all valuations of $k(B)$ corresponding to prime divisors on $B$. A conic bundle $\pi \colon X \to B$ determines a Brauer class $\alpha \in \Br(k(B))[2]$. Moreover, we have natural maps, for $\mathrm{char}(k)\neq 2$,
\[
\xymatrix{
\Br(k(B))[2] \ar[r]^{\oplus \partial_v\quad\quad\quad\quad\quad\quad\quad\quad}  & \bigoplus_{v \in B^{(1)}} H^1 (k(v), \Z /2) \simeq \bigoplus_{v \in B^{(1)}} k(v)\mult/k(v)\multtwo
}
\]
where $\partial_v$ are the usual residue maps as in, for example, \cite[Chaper 6]{GS06}, see also \cite[\S 3.1]{Pi16}; and for $\mathrm{char}(k)=2$,
\[
\xymatrix{
\Br(k(B))[2] \ar[r]^{\oplus r_v\quad\quad\quad\quad\quad\quad\quad\quad}  & \bigoplus_{v \in B^{(1)}} H^1 (k(v), \Z /2) \simeq \bigoplus_{v \in B^{(1)}} k(v)/\wp (k(v))
}
\]
where $r_v$ is the residue map as in Definition
\ref{dMerkurjevResidues}, provided it is defined for $\alpha$. In both
of these case, we call the image of $\alpha$ in $\bigoplus_{v \in
B^{(1)}} k(v)\mult/k(v)\multtwo$ in the first case, and in
$\bigoplus_{v \in B^{(1)}} k(v)/\wp (k(v))$ in the second case, the
residue profile of the conic bundle $\pi \colon X \to B$. Note that
the $v$'s for which the component in $H^1(k(v), \Z /2)$ of the residue
profile of a conic bundle is nontrivial are a (possibly proper) subset
of the divisorial valuations corresponding to the discriminant
components of the conic bundle.
\end{definition}

One main difference between characteristic not equal to $2$ and equal
to $2$ (besides the fact that the residue profiles are governed by
Kummer theory in the first case and by Artin--Schreier theory in the
second case) is the following: for $\mathrm{char}(k)\neq 2$ and $B$,
for concreteness and simplicity of exposition, a smooth projective
rational surface, the residue profiles of conic bundles that can occur
can be characterised as kernels of another explicit morphism, induced
by further residues; more precisely, there is an exact sequence
\begin{gather}\label{fBrauer}
\xymatrix{
0 \ar[r] & \Br(k(B))[2] \ar[r]^(.375){\oplus \partial_v}  & \bigoplus_{v \in
B^{(1)}} H^1 (k(v), \Z /2) \ar[r]^{\oplus\partial_p} & \bigoplus_{p
\in B^{(2)}} \mathrm{Hom}(\bbmu_2, \Z/2) 
}
\end{gather}
where $B^{(2)}$ is the set of codimension 2 points of $B$, namely, the
close points when $S$ is a surface, see \cite[Thm.~1]{A-M72},
\cite[Prop.~3.9]{Pi16}, but also the far-reaching generalization via
Bloch--Ogus--Kato complexes in \cite{Ka86}. The maps $\partial_p$ are
also induced by residues, more precisely, if $C\subset B$ is a curve,
$p\in C$ a point in the smooth locus of $C$, then
\[
\partial_p \colon H^1 (k(C), \Z/2) = k(C)\mult/k(C)\multtwo \to \mathrm{Hom}(\bbmu_2, \Z/2)_p \simeq \Z/2
\]
is just the valuation taking the order of zero or pole of a function
in $k(C)\mult/k(C)\multtwo$ at $p$, modulo $2$ (if $C$ is not smooth
at $p$, one has to make a slightly more refined definition involving
the normalisation).

One has the fundamental result of de Jong~\cite{deJ04},
\cite[Thm.~4.2.2.3]{Lieb08} (though for 2-torsion classes, it was
proved earlier by Artin~\cite[Thm.~6.2]{Artin82}) that for fields of
transcendence degree $2$ over an algebraically closed ground field $k$
(of any characteristic), the period of a Brauer class equals the
index, hence that every class in $\Br(k(B))[2]$ can be represented by
a quaternion algebra, i.e., by a conic bundle over an open part of
$B$.

However, in characteristic $2$, we cannot expect a sequence that
na{\"\i}vely has similar exactness properties as the one in
(\ref{fBrauer}), as the following example shows.

\begin{example}\label{eHC}
Let $X \subset \P^2 \times \P^2 \to \P^2$ be the conic bundle defined by an equation 
\[
	Q = ax^2 + a xz + by^2 + byz + c z^2=0,
\]
where $x,y,z$ are fibre coordinates in the ``fibre copy" $\P^2$ in $\P^2\times \P^2$, and $a, b, c$ are general linear forms in the homogeneous coordinates $u, v, w$ on the base $\P^2$. Then the discriminant is of degree $3$ and consists
of the three lines given by $a=0$, $b=0$ and $a=b$, which intersect in the point $P=(0:0:1)$. Indeed, if we want to find the points with coordinates $(u:v:w)$ on the base such that the fibre of the conic bundle over this point is singular, in other words, is such that there exist $(x:y:z)$ in $\P^2$ satisfying
\[
Q_x=az=0, \quad Q_y= bz=0, \quad  Q_z= by+ax= 0
\]
and also $ax^2 + axz + by^2 + byz + cz^2= 0$  (Euler's relation does not automatically imply the vanishing of the equation of the conic itself because the characteristic is two), then we have to look for those points $(u:v:w)$ where
\[
\begin{pmatrix}
a & b & 0\\
0 & 0 & a\\
0 & 0 & b\\
\sqrt{a} & \sqrt{b} & \sqrt{c}
\end{pmatrix}
\]
has rank less than or equal to $2$, which, on quick inspection, means $a=0, b=0$, or $a=b$.

More precisely, the conic bundle induces Artin--Schreier double covers ramified only in $P$ on each of those lines: For $a=0$ we have
\[
	Q_{a=0} = b (y^2 + yz) + c z^2
\]
which describes a nontrivial Artin--Schreier cover ramified only at $b=0$. 
The same happens on the line $b=0$ and also on the line $a=b$:
\begin{align*}
	Q_{a=b} 
	&= a(x^2+xz+y^2+yz) + c z^2 \\
	&=a\bigl( (x^2+y^2) + (x+y)z \bigr) + c z^2 \\
	&=a\bigl( (x+y)^2 + (x+y)z \bigr) + cz^2.
\end{align*}
\end{example}

The preceding example shows that we can indeed not expect a na{\"\i}ve
analogue of the sequence (\ref{fBrauer}) in characteristic $2$: to
define a reasonable further residue map to codimension 2 points, the
only thing that springs to mind here would be to assign some measure
of ramification at $P$ for each of the three Artin--Schreier
covers. But the resulting ramification measures would have to add to
zero (modulo $2$), and would have to be the same for each of the
covers, so that only the slightly ungeometric option to assign
ramification zero would remain. Note that the conic bundle in Example
\ref{eHC}, when lifted to characteristic $0$ by interpreting the
coefficients in the defining equation in $\Z$, has discriminant
consisting of the triangle of lines $a=0, b=0, 4c -a-b=0$, with double
covers over each of the lines ramified in the vertices of the
triangle. That might suggest that we should define a further residue
map also in characteristic $2$ by using local lifts to characteristic
$0$ and then summing the ramification indices in those points that
become identical when reducing modulo $2$, an idea that is reminiscent
of constructions in log geometry. But we have not succeeded in
carrying this out yet.

Moreover, the theory in \cite{Ka86}, although developed also in cases
where the characteristic equals the torsion order of the Brauer
classes under consideration, gives no satisfactory solution either
because the arithmetical Bloch--Ogus complex in \cite[\S 1]{Ka86} we
would need to study would be the one for parameters $i=-1, q=0$ and
then condition (1.1) in loc.cit. is not satisfied, whence the further
residue map we are looking for is undefined.

This seems to indicate that we have to do without a sequence such as (\ref{fBrauer}), and this is exactly what we will do in Section \ref{sFormulaBrauerChar2}: we will simply assume  existence of certain Brauer classes with predefined residue profiles, and we will prove this existence in practice, such as in the examples in Section \ref{sExamplesNontrivialBrauerChar2}, by writing down conic bundles over the bases under consideration that have the sought-for residue profiles. 

In fact, the next result partly explains Example \ref{eHC} and also shows that the situation in characteristic $2$ is even funnier.

\begin{theorem}\label{tDiscriminantProfilesChar2}
Let $\pi \colon X \to B$ be a conic bundle in characteristic $2$, where $B$ is again a smooth projective surface. Let $\Delta$ be its discriminant. Then there is no point $p$ of $\Delta$ locally analytically around which $\Delta$ consists of two smooth branches $\Delta_1$, $\Delta_2$ intersecting transversely in a point $p$ such that above $p$ the fibre of $X$ is a double line, and near $p$, the fibres over points in $\Delta_1\backslash \{p\}$ and $\Delta_2\backslash \{p\}$ are two intersecting lines in $\P^2$. 

On the other hand, in characteristic not two, the above is the generic local normal form of the discriminant of a conic bundle around a point above which the fibre is a double line.
\end{theorem}

\begin{proof}
Let $p \in \Delta$ be a point in the discriminant. Then, as in
Remark~\ref{rDiscriminants} and Definition~\ref{dDiscriminants}, let
$\P^2$ have homogeneous coordinates $(x:y:z)$ and $X_\univ
\longrightarrow \P$ be the universal conic bundle, and let $U \subset
B$ be a Zariski open neighborhood of $p$ such that $\Delta_i \cap U
\neq \emptyset$ for every irreducible component of $\Delta$ passing
through $p$, and such that there is a morphism $f \colon U \to \P$
realizing $\pi|_{\pi^{-1}(U)} \colon X_{\pi^{-1}(U)} \to U$ as
isomorphic to the pull-back via $f$ of the universal conic bundle.

Besides $\Delta_\univ \subset \P$, there is also the locus $\RR_1
\subset \P$ of double lines, defined for $\mathrm{char}(k)\neq 2$ by
the vanishing of the two by two minors of the associated symmetric
matrix yielding the generic conic (which coincides with the image of
the Veronese embedding $\P^2 \to \P^5=\P$), and for
$\mathrm{char}(k)=2$ by
\[
\RR_1 = \left\{  a_{xy} = a_{xz} = a_{yz} =0  \right\}.
\]
Let $f(p)=q$ and assume $q\in \RR_1$; after a coordinate change we can
assume (for all characteristics of $k$) that $q$ has coordinates
$a_{xx}=1$ and all other coordinates equal to zero. Expanding the
equation (\ref{fDisc2}) locally around the point $q$, we get (denoting
the dehomogenized affine coordinates by the same letters) the
following local equation of $\Delta_{\mathrm{univ}}$ around $q$ (which
becomes the origin in these affine coordinates)
\begin{gather}\label{fD2}
a_{yz}^2 + a_{xy}^2a_{zz} + a_{xy}a_{yz}a_{xz}+ a_{xz}^2 a_{yy}.
\end{gather}
The leading term is $a_{yz}^2$, whereas in characteristic not equal to $2$, the same procedure applied to (\ref{fDiscNot2}) yields
\begin{gather}\label{fDNot2}
\left( 4 a_{yy}a_{zz} - a_{yz}^2 \right) + a_{xy}a_{yz}a_{xz} - a_{xz}^2 a_{yy}  - a_{xy}^2 a_{zz}
\end{gather}
with leading term $4 a_{yy}a_{zz} - a_{yz}^2$. Now the discriminant $\Delta\cap U$ is given, in the characteristic $2$ case, by
\[
f^*(a_{yz})^2 + f^*(a_{xy})^2f^*(a_{zz}) + f^*(a_{xz})^2 f^*( a_{yy}) + f^*(a_{xy})f^*(a_{yz})f^*(a_{xz})
\]
showing that the projectived tangent cone to $\Delta$ at $p$ is either nonreduced of degree $2$ or has degree at least three (in Example \ref{eHC} the latter possibility occurs). This proves the  first assertion of the Theorem.

Also notice that the local normal form ruled out in characteristic $2$ by the above Theorem in a neighborhood of a point of the discriminant where the fibre is a double line, is in fact the generic local normal form in characteristic not equal to two! Indeed, by (\ref{fDNot2}), the tangent cone to $\Delta$ in $p$ is generically a cone over a nonsingular conic in $\P^1$, in other words, equal to two distinct lines. 
\end{proof}

\section{Nontriviality of the unramified Brauer group of a conic
bundle threefold in characteristic two}\label{sFormulaBrauerChar2}

We seek to prove the following result.

\begin{theorem}\label{tBrauerNontrivial}
Let $k$ be an algebraically closed field of characteristic $2$ and $B$
a smooth projective surface over $k$.  Let $\pi \colon X
\longrightarrow B$ be a conic bundle with discriminant $\Delta =
\cup_{i\in I} \Delta_i$ (as in Definition~\ref{dDiscriminants}) with
irreducible components~$\Delta_i$.  Suppose that the conic bundle is
tamely ramified over each $\Delta_i$ (as in Theorem
\ref{tResConicBundleAlongDiv}\textit{b)}); in particular, each 
$\Delta_i$ has multiplicity $1$.
Let $\alpha_i \in H^1
(k(\Delta_i), \Z /2) = k(\Delta_i)/\wp (k(\Delta_i))$ be the element
determined by the Artin--Schreier double cover induced by $\pi$ over $\Delta_i$.

Assume that one can write $I = I_1 \sqcup I_2$ with both $I_1, I_2$ nonempty such that:
\begin{enumerate}
\item There exists a conic bundle $\psi\colon Y\to B$ over $B$, or
possibly on a birational modification $B'\to B$, that induces a Brauer
class in $\Br(k(B))$ with residue profile (as in Definition
\ref{dResidueProfile}) given by $(\alpha_i)_{i\in I_1} \in
\bigoplus_{i\in I_1} H^1 (k(\Delta_i), \Z/2)$, and such that for any
point $P$ in the intersection of some $\Delta_i$ and $\Delta_j$, $i\in
I_1, j\in I_2$ (in case we work on some $B'$, this should hold for any
point $P$ lying over such an intersection), the fibre $Y_P$ is a cross
of two lines in $\P^2$.
\item
There exist $i_0\in I_1$ and $j_0\in I_2$ such that $\alpha_{i_0}$ and $\alpha_{j_0}$ are nontrivial.
\end{enumerate}
Then $\Br_\ur(k(X))[2]$ is nontrivial. 
\end{theorem}

Note that, by the discussion following Example \ref{eHC}, the assumption \textit{a)} seems hard to replace by something more cohomological or syzygy-theoretic.

\begin{proof}
Let us start the proof with a preliminary remark.  By the work of
Cossart and Piltant~\cite{CP08}, \cite{CP09}, resolution of
singularities is known for quasiprojective threefolds in arbitrary
characteristic.  (According to \cite{Hi17}, resolution of
singularities should always hold.) Then a smooth projective model
$\wt{X}$ of $X$ always exists and $\Br_\ur(k(X))[2] =\Br(\wt{X})[2]$
holds by Theorem~\ref{tComparisonBrauer}.  Still, in all applications,
for example in Section~\ref{sExamplesNontrivialBrauerChar2}, we will
always exhibit such a resolution explicitly.

By a result of Witt~\cite{Witt35}, cf.\ \cite[Thm. 5.4.1]{GS06}, the
kernel of the natural homomorphism
\[
\pi^* \colon \Br(k(B)) \to \Br(k(X))
\]
is generated by the class of the conic bundle $X \to B$ itself. Denote
by $\alpha$ that class in $\Br(k(B))$. Denote by $\beta$ the class of
$\psi \colon Y \to B$ in $\Br(k(B))$. We claim that $\pi^* (\beta) \in
\Br(k(X))$ is nontrivial and unramified. It is nontrivial because
$\beta \neq \alpha$ by assumption \textit{b)}: $\alpha$ and $\beta$
have different residues along some irreducible component
$\Delta_{j_0}$ of $\Delta$. Now to check that $\pi^* (\beta)$ is
unramified, it suffices to check that for any valuation $v=v_D$
corresponding to a prime divisor $D$ on a model $X'\to X$ which is
smooth generically along $D$ we have that $\pi^* (\beta)$ is
unramified with respect to that valuation, in the sense that it is in
the image of $\Br(\OO_{X', D})$. Let $\Delta^{(1)}= \bigcup_{i\in I_1}
\Delta_i$, $\Delta^{(2)}= \bigcup_{j\in I_2} \Delta_j$. There are two
cases to distinguish:
\begin{enumerate}
\item[(i)]
The centre $Z$ of $v$ on $B$, in other words the image of $D$ on $B$, is not contained in $\Delta^{(1)} \cap \Delta^{(2)}$. In general, notice that the in general only partially defined residue map is defined for the classes $\beta$ and $\alpha$ with respect to any divisor $D'$ on the base $B$ by the assumption on the geometric generic fibers of $X \to B$ over discriminant components and by Theorem \ref{tResConicBundleAlongDiv}. Moreover, if the centre $Z$ is not contained in $\Delta^{(1)} \cap \Delta^{(2)}$, then $\beta$ or $\beta - \alpha$ has residue zero along every divisor $D'$ on $B$ passing through $Z$. By Theorem \ref{tResMerkUnr}, the class $\beta - \alpha$ comes from $\Br(\OO_{B, Z})$. But $\pi^* (\beta -\alpha) = \pi^* (\beta)$, and hence $\pi^* (\beta )$ comes from $\Br(\OO_{X', D})$ as desired. 
\item[(ii)]
The centre $Z$ of $v$ on $B$ is contained in $\Delta^{(1)} \cap \Delta^{(2)}$, hence a point $Z=P$ over which the fibre $Y_P$ is a cross of lines by the assumption in \textit{a)} of the Theorem. Then the class $\pi^* (\beta)$ is represented by a conic bundle on $X'$ that has a split Artin--Schreier double cover as Merkurjev residue over $D$ by assumption \textit{a)} of the Theorem. Hence the residue in that case is defined along $D$ and trivial, so $\pi^* (\beta)$ comes from $\Br(\OO_{X', D})$ as desired by Theorem \ref{tResMerkUnr} again.
\end{enumerate}
Thus $\pi^*(\beta) \in \Br_\ur(k(X))[2]$ is a nontrivial class.
\end{proof}

\section{Examples of conic bundles in characteristic two with nontrivial Brauer groups}\label{sExamplesNontrivialBrauerChar2}

\begin{definition}\label{dConicBundleBo}
Consider the following symmetric matrix defined over $\Z$
\[
S = \begin{pmatrix}2 u v+4 v^{2}+2 u w+2 w^{2}&
      u^{2}+u w+w^{2}&
      u v\\
      u^{2}+u w+w^{2}&
      2 u^{2}+2 v w+2 w^{2}&
      u^{2}+v w+w^{2}\\
      u v&
      u^{2}+v w+w^{2}&
      2 v^{2}+2 u w+2 w^{2}\\
\end{pmatrix}.
\]
The bihomogeneous polynomial
\[
	(x,y,z)S(x,y,z)^t
\]
is divisible by $2$. Let $X \subset \P_{\Z}^2 \times \P_{\Z}^2$ be the conic bundle defined by
\[
	\frac{1}{2}(x,y,z)S(x,y,z)^t = 0.
\]
Here we denote by $(u:v:w)$ the coordinates of the first (base) $\P_{\Z}^2$ and by $(x:y:z)$
the coordinates of the second (fiber) $\P_{\Z}^2$.

\medskip

\noindent The determinant of $S$ is divisible by $2$ so 
\[
D=\frac{1}{2}\det S
\]
is still a polynomial over $\Z$.  Its vanishing defines the discriminant $\Delta$ of $X$ in the sense of Definition \ref{dDiscriminants}. We denote by $X_{(p)}$ the conic bundle over $\ol{\F}_p$ defined by reducing the defining equation of $X$ modulo $p$. It has discriminant $\Delta_{(p)}$ defined by the reduction of $D$ modulo $p$.
\end{definition}

This example was found using the computer algebra system {\tt Macaulay2} \cite{M2} and Jakob Kr\"okers Macaulay2 packages {\tt FiniteFieldExperiments} and {\tt BlackBoxIdeals} \cite{Kr15}.

Our aim is to prove the following result whose proof will take up the remainder of this Section.

\begin{theorem}\label{tInteresting}
The conic bundle $X \to \P^2$ has smooth total space that is not
stably rational over $\C$. More precisely, $X$ has the following
properties:
\begin{enumerate}
\item\label{iIrred}
The discriminant $\Delta_{(p)}$ is irreducible for $p \neq 2$, hence
$$\Br_\ur\left(\ol{\F}_p \left(X_{(p)}\right) \right)= 0\quad \text{for}\; p\neq 2.$$
\item\label{iBrauer2}
The conic bundle $X_{(2)}$ satisfies the hypotheses of Theorem \ref{tBrauerNontrivial}, hence $$\Br_\ur\left(\ol{\F}_2 \left(X_{(2)}\right) \right)[2]\neq 0.$$
\item\label{iChowZero}
There is a $\mathrm{CH}_0$-universally trivial resolution of singularities $\sigma\colon \wt{X}_{(2)}\to X_{(2)}$.
\end{enumerate}
\end{theorem}

Notice that the degeneration method of \cite{CT-P16} (and \cite{Voi15} initially, see also \cite{To16}) shows that \ref{iBrauer2}) and \ref{iChowZero}) imply that $X$ is not stably rational over $\C$: indeed, by Theorem \ref{tComparisonBrauer}, we have $\Br(\wt{X}_{(2)}) = \Br_\ur\left(\ol{\F}_2 \left(X_{(2)}\right) \right) \neq 0$, because of \ref{iBrauer2}); then \cite[Theorem 1.1]{ABBB18} yields that $\wt{X}_{(2)}$ is not $\mathrm{CH}_0$-universally trivial. Finally, \cite[Thm. 1.14]{CT-P16} implies that $X$ is not retract rational, in particular not stably rational, over $\ol{\Q}$, which is equivalent to saying it is not stably rational over any algebraically closed field of characteristic $0$, see \cite[Prop. 3.33]{KSC04}.

\medskip

Moreover, item \ref{iIrred}) shows that the degeneration method, using reduction modulo $p\neq 2$ and the unramified Brauer group, cannot yield this result. This follows from work of Colliot-Th\'{e}l\`{e}ne, see \cite[Thm 3.13, Rem. 3.14]{Pi16}; note that one only has to assume $X$ is a threefold which is nonsingular in codimension $1$ in \cite[Thm. 3.13]{Pi16}. Likewise, usage of differential forms as in \cite{A-O16}, see in particular their Theorem 1.1 and Corollary 1.2, does not imply the result either.

\subsection{Irreducibility of $\Delta_{(p)}$ for $p\neq 2$}\label{ssDiscNot2}
When we speak about irreducibility or reducibility in the following, we always mean geometric irreducibility or reducibility. 
Our first aim is to prove that $\Delta_{(p)}$ is irreducible for $p \not= 2$. This is easy for generic $p$ since $X$ is smooth over
$\Q$ (by a straight-forward Gr\"obner basis computation  \cite{ABBBM2}). Since being singular is a codimension $1$ condition, we expect that $\Delta_{(p)}$ is singular for a finite number of primes. So we need a more refined argument to prove irreducibility. Our idea is to prove that there is (counted with multiplicity) at most one singular point for each $p \not=2$. 

\begin{lemma}\label{lFivePoints}
Let $C$ be a reduced and reducible plane curve of degree at least $3$ over an algebraically closed field. Then the length of the singular subscheme, defined by the Jacobi ideal on the curve, is at least $2$.
\end{lemma}

\begin{proof}
The only singularities of length $1$ are those where \'{e}tale locally two smooth branches of the curve cross transversely: if $f(x,y)=0$ is a local equation for $C$ with isolated singular point at the origin, then the length can only be $1$ if 
\[
\frac{\partial f}{\partial x}, \quad \frac{\partial f}{\partial y}
\]
have leading terms consisting of linearly independent linear
forms. This means two smooth branches cross transversely. The only
reducible curve that has only one transverse intersection is the union
of two lines.
\end{proof}

We also need the following technical lemma.

\begin{lemma} \label{lAtMostOne}
Let $I \subset \Z[u,v,w]$ be a homogeneous ideal, $B = \{l_1,\dots,l_k\}$ a $\Z$-Basis of the space of linear forms $I_1 \subset I$,
and $M$ the $k \times 3$ matrix of coefficients of the $l_i$. Let $g$ be the minimal generator of the ideal of $2 \times 2$ minors of $M$ in $\Z$.

If a prime $p$ does not divide $g$, then $I$ defines a finite scheme of degree at most $1$ in characteristic $p$.
\end{lemma}

\begin{proof}
If $p$ does not divide $g$ there is at least one $2 \times 2$ minor $m$ with $p\nmid m$. Therefore in characteristic $p$ this minor is invertible and the matrix has rank at least $2$. It follows that $I$ contains at least $2$ independent linear forms in characteristic $p$ and therefore the vanishing set is either empty or contains $1$ reduced point. 
\end{proof}

\begin{remark}
Notice that the condition $p \nmid g$ is sufficient, but not necessary. For example the ideal $(u^2,v^2,w^2)$ vanishes nowhere,
but still has $g=0$ and therefore $p | g$. The condition becomes necessary if $I$ is saturated. 
\end{remark}

\begin{proposition}\label{pIrredDelta}
For $p \not= 2$, $\Delta_{(p)}$ is an irreducible sextic curve.
\end{proposition}

\begin{proof}
We apply Lemma \ref{lAtMostOne} to the saturation of
$(D,\frac{dD}{du},\frac{dD}{dv},\frac{dD}{dw}) \subset \Z[u,v,w]$.  A
Macaulay2 computation gives $g=2^{10}$ \cite{ABBBM2}. So we have at
most one singular point over $p \not= 2$ and therefore $\Delta_{(p)}$
is irreducible.
\end{proof}

\subsection{The unramified Brauer group of $X_{(2)}$ is nontrivial}\label{ssBrauer Nontrivial}

Let us now turn to characteristic $p=2$.

\begin{proposition}\label{pDiscDescription}
We have
\[
	D \equiv uw(u+w)(\gamma u + v^3) \mod 2.
\]
with $\gamma = v^2+uv+vw+w^2$. Furthermore
\begin{itemize}
\item $\gamma$ does not vanish at $(0:0:1)$.
\item $\gamma u + v^3 =0$ defines a smooth elliptic curve $E \subset \P_{\ol{\F}_2}^2$.
\item $E$ does not contain the intersection point $(0:1:0)$ of the three lines.
\item The intersection of $E$ with each of the lines $w=0$ and $u+w=0$ is transverse.
\item The line $u=0$ is an inflectional tangent to $E$ at the point $(0:0:1)$. 
\end{itemize}
\end{proposition}

\begin{proof}
All of this is a straight forward computation. See \cite{ABBBM2}.
\end{proof}

The next lemma gives us a criterion for the irreducibility of the Artin--Schreier double covers induced
on the discriminant components and hence for the nontriviality of the residues of the conic bundle along these components.

\begin{lemma}\label{lCoverHC}
Let $\pi\colon \ol{X} \to \P^2$ be a conic bundle defined over $\F_2$. Let
$C \subset \P^2$ be an irreducible curve  over $\F_2$, over which
the fibers of $\ol{X}$ generically consist of two distinct lines.  Let
$\wt{C}\to C$ be the natural double cover of $C$ induced by
$\pi$. Then $\wt{C}$ is irreducible if the following hold: 
\begin{itemize}
\item There exists an $\F_2$-rational point $p_1 \in C$ such that the fiber of
$\ol{X}$ over $p_1$ splits into two lines defined over $\F_2$.
\item There exists an $\F_2$-rational point $p_2 \in C$ such that the fiber of
$\ol{X}$ over  $p_2$ is irreducible over $\F_2$ but splits into two lines over $\ol{\F}_2$.
\end{itemize}

\end{lemma}
\begin{proof}
Under the assumptions the double cover $\wt{C} \to C$ is defined
over $\F_2$. Suppose, by contradiction, that $\wt{C}$ were
(geometrically) reducible. Then the Frobenius morphism $F$ would
either fix each irreducible component of $\wt{C}$ as a set, or
interchange the two irreducible components. But since $C$ is defined
over $\F_2$, this would mean that $F$ either fixes each of the two
lines as a set in every fiber over a $\F_2$-rational point of the base,
or $F$ interchanges the two lines in every fiber over a $\F_2$-rational
point. This contradicts the existence of $p_1, p_2$.
\end{proof}

\begin{proposition}\label{pFibersExample}
We consider the fibers of $X_{(2)}$ over the $\F_2$-rational points of the base $\P_{\ol{\F}_2}^2$
and obtain the following table:

\begin{center}
\begin{tabular}{|c|c|c|c|c|c|c|}
\hline
point & fiber & $u$ & $w$ & $u+w$ & $\gamma u +v^3$ \\
\hline
$(0:1:0)$ & $1$ double line 		& $\times$ & $\times$ & $\times$ &  \\
$(0:1:1)$ & $2$ rational lines		& $\times$ &    &    &\\
$(1:0:0)$ & $2$ rational lines		&    & $\times$ &    &$\times$ \\
$(1:0:1)$ & $2$ rational lines		&    &    & $\times$ &\\
$(0:0:1)$ & $2$ conjugate lines 	& $\times$ &    &    & $\times$\\
$(1:1:0)$ & $2$ conjugate lines	&    & $\times$ &    &\\
$(1:1:1)$ & $2$ conjugate lines	&    &    & $\times$ &\\
\hline
\end{tabular}
\end{center}

\noindent
Here if a $\F_2$-rational point lies on a particular component of the
discriminant, we put an '$\times$' in the corresponding row and
column.
\end{proposition}

\begin{proof}
All of this is again a straight forward computation. See  \cite{ABBBM2}.
\end{proof}

\begin{corollary}\label{cCoversExample}
The conic bundle $X_{(2)}$ induces a nontrivial Artin--Schreier double cover on each component of the discriminant locus. In particular, condition \textit{b)} of Theorem \ref{tBrauerNontrivial} is satisfied if we let $I_1$ index the three lines of $\Delta_{(2)}$ and $I_2$ the elliptic curve $E$.
\end{corollary}

\begin{proof}
Use Lemma \ref{lCoverHC} and Proposition \ref{pFibersExample}.
\end{proof}

We now want to check that the $2:1$ covers induced by $X_{(2)}$ over the three lines yield the same element in $H^1 (\ol{\F}_2 (t), \Z /2)$ as the $2:1$ covers in our Example \ref{eHC}. We need this to verify condition \textit{a)} of Theorem \ref{tBrauerNontrivial}. We only have to check that all these covers are birational to each other over the base $\P^1$. For this we use:

\begin{proposition}\label{pArtinSchreierCoversExample}
We work over the ground field $k=\ol{\F}_2$.  
Let $\ol{X} \subset \P^2 \times \P^2$ be a divisor of bidegree $(d,2)$ that is a conic bundle over $\P^2$ via the first projection. 
Let furthermore
$L \subset \P^2$ be a line in the discriminant of $\ol{X}$ such that $\ol{X}$ defines an Artin--Schreier double cover of $L$ branched in 
a single reduced point $R$. Here the scheme-structure on $R$ is defined by viewing it as the scheme-theoretic pull-back of the locus of double lines in the universal discriminant as in Definition \ref{dDiscriminants}. Suppose also that $\ol{X}$, $L$ and $R$ are defined over $\F_2$.

\medskip

Then either the $2:1$ cover defined over $L$ by $\ol{X}|_L$ is trivial, or it is birational over $L$ to the Artin--Schreier cover
\[
	x^2 + x + \frac{v}{u} = 0
\]
where $(u:v)$ are coordinates on $L \cong \P^1$ such that $R = (0:1)$. In particular all non-trivial covers with $R=(0:1)$ that satisfy the above conditions yield the same element in $H^1 (\ol{\F}_2 (t), \Z /2)$. 
\end{proposition}

\begin{proof}
Note that $\ol{X}|_L \to L$ is defined over $\F_2$, and defines a double cover $\pi\colon Y \to L$, where $Y$ is the relative Grassmannian of lines in the fibres of $\ol{X}|_L \to L$. Then $Y\to L$ is also defined over $\F_2$ and flat over $L$. Hence $\EE = \pi_* (\OO_Y)$ is a rank $2$ vector bundle on $L$, and $Y$ can be naturally embedded into $\P (\EE)$. Then $\EE=\OO_L (e) \oplus \OO_L (f)$, and $Y$ is defined inside $\P (\EE )$ by an equation
\[
	a x^2 + b xy + cy^2=0
\]
with $a, b, c$ homogeneous polynomials with $\deg (a) +\deg (c) = 2 \deg (b)$. Notice that $b=0$ defines the locus of points of the base $L$ over which the fibre is a double point. By our assumption $b=0$ is a single reduced point. Hence $\deg (b) =1$. 

Notice that if the double cover $Y$ is nontrivial, both $a$ and $c$ are nonzero, hence $\deg (a) \ge 0, \deg (c) \ge 0$ and $\deg (a) +\deg (c) = 2$. 

\medskip

Let $(u:v)$ be homogeneous coordinates on $L$.
We now put $a'=a/u^{\deg (a)}, b' = b/u^{\deg (b)}, c' = c/u^{\deg{c}}$ and calculate over the function field of $L$.  
Apply $(x,y) \mapsto (b'x,a'y)$ to obtain
\[
	a' (b')^2 x^2 + a'(b')^2 xy + (a')^2c'y^2=0.
\]
Divide by $a'(b')^2$ and dehomogenise via $y\mapsto 1$ to obtain the Artin--Schreier normal form
\[
	x^2 + x + \frac{ac}{b^2}=0.
\]
We now use the fact that we can choose the coordinates $(u:v)$ such that $b=u$. We can write $ac = \alpha u^2 + \beta uv + \gamma v^2$ with $\alpha, \beta, \gamma \in \F_2$:
\[
	x^2 + x + \alpha + \beta\frac{v}{u} + \gamma \frac{v^2}{u^2}=0.
\]
Now we use extensively the fact that we work over $\F_2$:
firstly, either $\alpha=0$ or $\alpha =1$. In the second case let $\rho \in \ol{\F}_2$ be a root of $x^2+x+1$ and apply $x \mapsto x+\rho$. This gives 
\[
	x^2 + x + \beta\frac{v}{u} + \gamma \frac{v^2}{u^2} =0
\]
in  both cases. 
Even though the transformation was defined over $\ol{\F}_2$ this does not change the fact that $\beta$ and $\gamma$
are in $\F_2$. 

Secondly either $\gamma = 0$ and we have
\[
	x^2 + x + \beta \frac{v}{u} = 0
\]
or $\gamma = 1$ and we apply $x\mapsto x+\frac{v}{u}$ to obtain
\[
	x^2 + x +  (\beta+1)\frac{v}{u} = 0.
\] 
In both cases the coefficient in front of $\frac{v}{u}$ is either $0$ or $1$, thus the cover is either trivial or has the normal form
\[
	x^2 + x + \frac{v}{u} = 0.
\] 
\end{proof}

\begin{remark}
Notice that the proof works over any field $k$ of characteristic $2$ until we have
\[
	x^2 + x + \beta\frac{v}{u} + \gamma \frac{v^2}{u^2} = 0.
\]
Now we can eliminate $\gamma$ only if it is a square in $k$. Even if this happens (for example if we work over $\ol{\F}_2$) we obtain, using $x \mapsto x + \sqrt{\gamma}(v/u)$, 
\[
x^2 + x + \underbrace{\Bigl(\beta + \sqrt{\gamma} \Bigr)}_{\beta'} \frac{v}{u} = 0.
\]
So there seems to be a $1$-dimensional moduli space of such covers.
\end{remark}

Note that Propositions \ref{pDiscDescription}, \ref{pFibersExample},
Corollary \ref{cCoversExample}, and Proposition
\ref{pArtinSchreierCoversExample} together with the conic bundle
exhibited in Example \ref{eHC} show that Theorem
\ref{tBrauerNontrivial} is applicable in the case of $X_{(2)}$, hence
$\Br_\ur(\ol{\F}_2 (X_{(2)}))[2] \neq 0$.

\subsection{A $\mathrm{CH}_0$-universally trivial resolution of $X_{(2)}$}\label{ssChowZeroRes}
Now we conclude the proof of Theorem \ref{tInteresting} by showing the
remaining assertion \textit{c)}, the existence of a
$\mathrm{CH}_0$-universally trivial resolution of singularities
$\sigma \colon \wt{X}_{(2)} \to X_{(2)}$.

We will use the following criterion \cite[Ex.~2.5~(1),(2),(3)]{Pi16} which summarizes results of \cite{CT-P16} and \cite{CTP16-2}.

\begin{proposition}\label{pChowTrivCrit}
A sufficient condition for a projective morphism $f \colon V \to W$ of
varieties over a field $k$ to be $\mathrm{CH}_0$-universally trivial
is that the fibre $V_{\xi}$ of $f$ over every scheme-theoretic point
$\xi$ of $W$ is a (possibly reducible) $\mathrm{CH}_0$-universally
trivial variety over the residue field $\kappa (\xi)$ of the point
$\xi$. This sufficient condition in turn holds if $X_{\xi}$ is a
projective (reduced) geometrically connected variety, breaking up into
irreducible components $X_i$ such that each $X_i$ is
$\mathrm{CH}_0$-universally trivial and geometrically irreducible, and
such that each intersection $X_i\cap X_j$ is either empty or has a
zero-cycle of degree $1$ (of course the last condition is automatic if
$\kappa (\xi)$ is algebraically closed).

Moreover, a smooth projective retract rational variety $Y$ over any field is universally $\mathrm{CH}_0$-trivial. If $Y$ is defined over an algebraically closed ground field, one can replace the smoothness assumption on $Y$ by the requirement that $Y$ be connected and each component of $Y^{\mathrm{red}}$ be a rational variety with isolated singular points.  \end{proposition}

We now study the behaviour of $X_{(2)}$ locally above a point $P$ on the base $\P^2$, distinguishing several cases; for the cases when $X_{(2)}$ is singular locally above $P$, we exhibit an explicit blow-up scheme to desingularise it, with exceptional locus $\mathrm{CH}_0$-universally trivial, so that Proposition \ref{pChowTrivCrit} applies. 

\medskip

\noindent \textbf{(i) The case when $P \not\in \Delta_{(2)}$.} In that case, $X_{(2)}$ is nonsingular locally above $P$.

\medskip

\noindent \textbf{(ii) The case when $P$ is in the smooth locus of $\Delta_{(2)}$.} In that case, $X_{(2)}$ is nonsingular locally above $P$ as well. This can be seen by direct computation \cite{ABBBM2}.

\medskip

\noindent \textbf{(iii) The case when $P= (0:1:0)$ is the intersection point of the three lines $(u=0)$, $(w=0)$, $(u+w=0)$ in $\Delta_{(2)}$.} A direct computation shows that here $X_{(2)}$ is nonsingular locally above $P$ as well \cite{ABBBM2}.

\medskip

\noindent \textbf{(iv) The case when $P$ is one of the six intersection points of $w=0$ or $u+w=0$ with $E$.} In these cases, the intersection of the two discriminant components is transverse, and the fibre above the intersection point is a cross of lines. Then $X_{(2)}$ locally above $P$ has a $\mathrm{CH}_0$-universally trivial desingularization because we have the local normal form as in Lemma \ref{lNormalFormConicExample} with $n=1$, and thus, by Proposition \ref{pBlowUpScheme} and Proposition \ref{pExceptionalDivisors}, one blow-up with exceptional divisor a smooth quadric resolves the single singular point of $X_{(2)}$ above $P$.

\medskip

\noindent \textbf{(v) The case when $P=(0:0:1)$ is the point where the components $(u=0)$ and $E$ of $X_{(2)}$ intersect in such a way that $(u=0)$ is an inflectional tangent to the smooth elliptic curve $E$.} In this case, the fibre of $X_{(2)}$ above $P$ is a cross of two conjugate lines by Proposition \ref{pFibersExample}. We need some auxiliary results.

\begin{lemma}\label{lNormalFormConicExample}
Let $\hat{\A}^2_{\ol{\F}_2}$ be the completion of $\A^2_{\ol{\F}_2}$ with affine coordinates $u, v$ along $(0,0)$, and let $\ol{X}$ be a conic bundle over $\hat{\A}^2_{\ol{\F}_2}$. Thus $\ol{X}$ has an equation
\[
	c_{xx} x^2 + c_{xy} xy + c_{yy} y^2 + c_{xz} xz + c_{yz} yz + c_{zz} z^2 = 0
\]
where the $c$'s are formal power series in $u$ and $v$ with coefficients in $\ol{\F}_2$. 

Assume that 
\begin{enumerate}
\item \label{iDegThree} locally around $(0,0)$ the discriminant of $\ol{X}$ has a local equation $u(u+v^n)$, $n\ge 1$.
\item \label{iNonTrivial}The fiber over $(0,0)$ has the form $x^2+xy+y^2$
\end{enumerate}

Then, after a change in the fibre coordinates $x, y$ and $z$, we can assume the normal form
\[
	x^2+xy+c_{yy}y^2 + c_{zz}z^2 = 0
\]
with $c_{yy}$ a unit, $c_{zz} = \beta u(u+v^n)$ and $\beta$ a unit.
\end{lemma}

\begin{proof}
Because of assumption $(\ref{iNonTrivial})$ we can assume that $c_{xx}$ is a unit. After dividing by $c_{xx}$ we 
can assume that we have the form
\[
	x^2 + c_{xy} xy + c_{yy} y^2 + c_{xz} xz + c_{yz} yz + c_{zz} z^2 = 0
\]
with $c_{xy}$ and $c_{yy}$ units. After the substition of $x\mapsto c_{xy} x$ we can divide the whole equation
by $c_{xy}^2$ and can assume that we have the form
\[
	x^2 + xy + c_{yy} y^2 + c_{xz} xz + c_{yz} yz + c_{zz} z^2 = 0
\]
with $c_{yy}$ a unit. Now substituting $x\mapsto x+c_{yz}z$ and $y \mapsto y+c_{xz}z$ we obtain the normal 
form
\[
	x^2 + xy + c_{yy} y^2 + c_{zz} z^2 = 0
\]
with $c_{yy}$ still a unit. Now the discriminant of this conic bundle ist $c_{zz}$. Since the discriminant was changed at most
by a unit during the normalization process above, we have $c_{zz} = \beta u(u+v^n)$ as claimed.
\end{proof}

\begin{proposition}\label{pBlowUpScheme}
Let 
$Y$ be a hypersurface in $\hat{\A}^4_{\ol{\F}_2}$ with coordinates $x,y,u,v$, with equation
\[
   x^2 + xy + \alpha y^2 + \beta u(u+v^n) = 0, \quad n\ge 1,
\]
where $\alpha$ and $\beta$ are units in $\ol{\F}_2 [[u,v]]$. Then $Y$ is singular only at the origin. 

Let $\widetilde{\A^4}$ be the blow up of
$\hat{\A}^4_{\ol{\F}_2}$ in the origin and let $\widetilde{Y} \subset \widetilde{\A^4}$ be the strict transform of $Y$. If $n=1$, then $\widetilde{Y}$ is smooth. If $n>1$, then $\widetilde{Y}$ is singular at only one point, which we can assume to be the origin again.
Around this singular point $\widetilde{Y}$ has  a local equation
\[
   x^2 + xy + \alpha' y^2 + \beta' u(u+v^{n-1}) = 0
\]
with $\alpha'$ and $\beta'$ units in $\ol{\F}_2[[u,v]]$.
\end{proposition}

\begin{proof}
In $\widetilde{\A^4}$, we obtain $4$ charts. It will turn out that in three of them $\widetilde{Y}$ is smooth and in the fourth we obtain the local equation given above. 

\begin{enumerate}
\item $(x,y,u,v) \mapsto (x,xy,xu,xv)$ gives
\[
   x^2 + x^2y + \alpha' x^2y^2 + \beta' xu(xu+x^nv^n) = 0
\]
as the total transform, and 
\[
   1 + y + \alpha' y^2 + \beta' u(u+x^{n-1}v^n)  = 0
\]
as the strict transform. Notice that $\alpha'$ and $\beta'$
are power series that only involve $u,v$ and $x$. Therefore the derivative with respect
to $y$ is $1$ in both cases and the strict transform is
smooth in this chart.

\item $(x,y,u,v) \mapsto (xy,y,yu,yv)$ gives
\[
   x^2y^2 + xy^2 + \alpha' y^2 + \beta' yu(yu+y^nv^n) = 0
\]
as the total transform, and 
\[
   x^2 + x + \alpha' + \beta' u(u+y^{n-1}v^n)  = 0
\]
as the strict transform. Notice that $\alpha'$ and $\beta'$
are power series that only involve $u,v$ and $y$. Therefore the derivative with respect
to $x$ is $1$ in both cases and the strict transform is
smooth in this chart. 

\item $(x,y,u,v) \mapsto (xu,yu,u,uv)$ gives
\[
x^2u^2 + xyu^2 + \alpha' y^2u^2 + \beta' u(u+u^nv^n)  = 0
 \]
as the total transform, and 
\[
x^2 + xy + \alpha' y^2 + \beta' (1+u^{n-1}v^n)  = 0
   \]
as the strict transform. Notice that $\alpha'$ and $\beta'$
are power series that only involve $u,v$. Therefore the derivative with respect
to $x$ and $y$ are $y$ and $x$ respectively. So the singular locus lies on $x=y=0$. Substituting
this into the equation of the strict transform we get
\[
	\beta' (1+u^{n-1}v^n)  = 0
\]
This is impossible since $\beta'$ and $(1+u^{n-1}v^n)$ are units. Therefore the strict transform
is smooth in this chart.

\item $(x,y,u,v) \mapsto (xv,yv,uv,v)$ gives
\[
	 x^2v^2 + xyv^2 + \alpha' y^2v^2 + \beta' uv(uv+v^n) = 0
 \]
as the total transform, and 
\[
	 x^2 + xy + \alpha' y^2 + \beta' u(u+v^{n-1}) = 0
 \]
as the strict transform.  Notice that $\alpha'$ and $\beta'$
are power series that only involve $u,v$. Therefore the derivative with respect
to $x$ and $y$ are $y$ and $x$ respectively. So the singular locus lies on $x=y=0$. Substituting
this into the equation of the strict transform we get
\[
	\beta' u(u+v^{n-1}) = 0.
\]
Let us now look at the derivative with respect to $u$:
\[
	\frac{d\alpha'}{du} y^2 +  \frac{d\beta'}{du} u(u+v^{n-1}) + \beta'  v^{n-1} =0 
\]
Since $x=y=u(u+v^{n-1})=0$ on the singular locus, this equation reduces to $v^{n-1}=0$. If $n=1$, this shows that $\widetilde{Y}$ is smooth everywhere. If $n\ge 2$, we obtain that the strict transform is singular at most at $x=y=u=v=0$ in this chart. To check that this
is indeed a singular point we also calculate the derivative with respect to $v$:
\[
	\frac{d\alpha'}{dv} y^2 +  \frac{d\beta'}{dv} u(u+v^{n-1})  + \beta' (n-1)uv^{n-2} =0 
\]
which is automatically satisfied at $x=y=u=v=0$. 
\end{enumerate}

This proves all claims of the proposition. 
\end{proof}

\begin{proposition}\label{pExceptionalDivisors}
Keeping the notation of Proposition \ref{pBlowUpScheme}, the exceptional divisor of $\widetilde{Y} \to Y$ is a quadric with at most one singular point. 
\end{proposition}

\begin{proof}
Recall that the equation of $Y$ is
\[
   x^2 + xy + \alpha y^2 + \beta u(u+v^n) = 0.
\]
We see immediately that the leading term around the origin is
\[
x^2 + xy + \alpha_0 y^2 + \beta_0 u^2
\]
for $n>1$ with $\alpha_0$, $\beta_0$ nonzero constants, and
\[
x^2 + xy + \alpha_0 y^2 + \beta_0 u^2 +\beta_0 uv
\]
for $n=1$. The first is a quadric cone with an isolated singular point, the second is a smooth quadric.
\end{proof}

Summarizing, we see that Lemma \ref{lNormalFormConicExample},
Proposition \ref{pBlowUpScheme}, and Proposition~\ref{pExceptionalDivisors} show that, locally around the singular
point lying above $P=(0:0:1)$, the conic bundle $X_{(2)}$ has a
resolution of singularities with $\mathrm{CH}_0$-universally trivial
fibres. By Proposition \ref{pChowTrivCrit}, and taking into account
cases (i)-(v) above, we conclude that $X_{(2)}$ has a
$\mathrm{CH}_0$-universally trivial resolution of singularities
$\sigma \colon \wt{X}_{(2)} \to X_{(2)}$. This concludes the proof of
Theorem \ref{tInteresting}.

\providecommand{\bysame}{\leavevmode\hbox to3em{\hrulefill}\thinspace}
\providecommand{\href}[2]{#2}


\begin{thebibliography}{99999999999}
\bibitem[A-O16]{A-O16}
H.\ Ahmadinezhad, T.\ Okada, T., \emph{Stable rationality of higher dimensional conic bundles}, preprint (2016) \href{https://arxiv.org/abs/1612.04206}{arXiv:1612.04206 [math.AG]}.


\bibitem[Artin67]{Artin67}
E.\ Artin, \emph{Algebraic Numbers and Algebraic Functions}, Gordon and Breach, New York (1967); reissued AMS Chelsea Publishing, AMS, (2006).

\bibitem[Artin07]{Artin07}
E.\ Artin, \emph{Algebra with Galois Theory}, Courant Lecture Notes \textbf{15}, reprint, AMS (2007).

\bibitem[Artin82]{Artin82}
M.\ Artin,
\textit{Brauer--{S}everi varieties},
Brauer groups in ring theory and algebraic geometry ({W}ilrijk, 1981),
Lecture Notes in Math., vol.\ 917, pp.\ 194--210, Springer, Berlin, 1982.

\bibitem[A-M72]{A-M72}
M.\ Artin, D.\ Mumford, \emph{Some elementary examples of unirational varieties which are not rational}, Proc. London Math. Soc. (3) \textbf{25}, (1972), 75--95.

\bibitem[ACTP16]{ACTP16}
A.~Auel, J.-L.~Colliot-Th\'el\`ene, and Parimala,
\textit{Universal unramified cohomology of cubic fourfolds containing
a plane}, Brauer groups and obstruction problems: moduli
spaces and arithmetic (Palo Alto, 2013), 
Progress in Mathematics, vol.\ 320, Birkh\"auser Basel, 2017.

\bibitem[ABBP16]{ABBP16}
A.\ Auel, Chr.\ B\"ohning, H.-Chr.\ Graf\,v.\,Bothmer, A.\ Pirutka, \emph{Conic bundles with nontrivial unramified Brauer group over threefolds}, preprint (2016) \href{https://arxiv.org/abs/1610.04995}{arXiv:1610.04995 [math.AG]}.

\bibitem[ABBB18]{ABBB18}
A.\ Auel,  A.\ Bigazzi, Chr.\ B\"ohning, H.-Chr.\ Graf\,v.\,Bothmer,
\emph{The universal triviality of the Chow group of $0$-cycles and the
Brauer group}, preprint (2018) \href{https://arxiv.org/abs/1806.02676}{arXiv:1806.02676 [math.AG]}.

\bibitem[ABBBM2]{ABBBM2}
A. \ Auel,  A.\ Bigazzi, Chr.\ B\"ohning, H.-Chr.\ Graf\,v.\,Bothmer, \emph{Macaulay2 scripts for ``Unramified Brauer groups of conic bundles over rational surfaces in characteristic $2$"}, available at \url{http://www.math.uni-hamburg.de/home/bothmer/m2.html}.

\bibitem[AB68]{AB68}
M. \ Auslander and A.\ Brumer, \emph{Brauer groups of discrete valuation rings}, Nederl. Akad. Wetensch. Proc. Ser. A \textbf{71} (1968), 286--296.

\bibitem[AG60]{AG60}
M.\ Auslander and O.\ Goldman, \emph{Maximal Orders}, Transactions of
the American Math. Soc. \textbf{97} (1960), no.\ 1, 1--24.



\bibitem[Bogo87]{Bogo87}
F.\ Bogomolov, \emph{The Brauer group of quotient spaces of linear
representations}, Izv. Akad. Nauk SSSR Ser. Mat. \textbf{51} (1987), no. 3, 485--516.


\bibitem[CL98]{CL98}
A.\ Chambert-Loir, \emph{Cohomologie cristalline: un survol}, Exposition. Math. \textbf{16}, (1998) 333--382.

\bibitem[Ces17]{Ces17}
K.\, Cesnavicius, \emph{Purity for the Brauer group}, preprint (2017) \href{https://arxiv.org/abs/1711.06456}{arXiv:1711.06456v3 [math.AG]}.

\bibitem[CTO]{CTO}
J.-L.\ Colliot-Th\'{e}l\`{e}ne, M.\ Ojanguren, \emph{Vari\'{e}t\'{e}s unirationelles non rationelles: au del\`{a} de l'exemple d'Artin et Mumford}, Invent. Math. \textbf{97} (1989), 141--158 .

\bibitem[CT95]{CT95}
J.-L.\ Colliot-Th\'{e}l\`{e}ne,  \emph{Birational invariants, purity and the Gersten conjecture}, in: K-theory and Algebraic Geometry: connections with quadratic forms and division algebras (Santa Barbara, CA, 1992) (B. Jacob and A. Rosenberg, eds.), Proceedings of Symposia in Pure Mathematics \textbf{58.1} (1995), 1--64.

 
\bibitem[CT99]{CT99}
J.-L.\ Colliot-Th\'{e}l\`{e}ne, \emph{Cohomologie des corps valu\'{e}s hens\'{e}liens, d'apr\`{e}s K. Kato et S. Bloch}, in:  Algebraic $K$-theory and its applications (Trieste, 1997), Proceedings of the Workshop and Symposium, ICTP, Trieste, Italia, 1-19 September 1997, ed. H. Bass, A. O. Kuku, C. Pedrini, World Scientific Publishing, (1999), 120--163.

\bibitem[CTPS12]{CTPS12}
J.-L.\ Colliot-Th\'{e}l\`{e}ne, R. \ Parimala, V.\ Suresh, \emph{Patching and local-global principles for homogeneous spaces over function fields of $p$-adic curves},  Comment. Math. Helv. \textbf{87} (2012) 1011--1033

\bibitem[CT-P16]{CT-P16}
J.-L.\ Colliot-Th{\'e}l{\`e}ne and A.\ Pirutka,
\emph{Hypersurfaces quartiques de dimension 3 : non rationalit{\'e} stable},
Ann. Sci. {\'E}cole Norm. Sup. \textbf{49} (2016), no. 2, 371--397.

\bibitem[CTP16-2]{CTP16-2} 
J.-L.\ Colliot-Th{\'e}l{\`e}ne and A.\ Pirutka,
\emph{Cyclic covers that are not stably rational}, (in Russian) Izvestija RAN, Ser. Math. Tom \textbf{80} no. 4  (2016), 35--48; English translation available at \href{https://www.math.u-psud.fr/~colliot/CTPircyclique_anglais.pdf}{here}. 

\bibitem[CP08]{CP08}
V.\ Cossart and O.\ Piltant, 
\textit{Resolution of singularities of threefolds in positive
characteristic. I. Reduction to local uniformization on
Artin--Schreier and purely inseparable coverings},
J. Algebra \textbf{320} (2008), no.\ 3, 1051--1082.

\bibitem[CP09]{CP09}
V.\ Cossart and O.\ Piltant, 
\textit{Resolution of singularities of threefolds in positive
characteristic. II},
J. Algebra \textbf{321} (2009), no.\ 7, 1836--1976.

\bibitem[deJ03]{deJ03}
A. J.\ de Jong, \emph{A result of Gabber}, preprint (2003), available at \href{http://www.math.columbia.edu/~dejong/papers/2-gabber.pdf}{http://www.math.columbia.edu/~dejong/papers/2-gabber.pdf}

\bibitem[deJ04]{deJ04}
A.J.\ de Jong,
\textit{The period-index problem for the {B}rauer group of an
algebraic surface},
Duke Math. J. \textbf{123} (2004), no. 1, 71--94.

\bibitem[EKM08]{EKM08}
R., Elman, N. Karpenko, A. Merkurjev, \emph{The Algebraic and Geometric Theory of Quadratic Forms}, AMS Colloquium Publications, vol. \textbf{56}, AMS (2008).

\bibitem[Fuji02]{Fuji02}
K.\ Fujiwara, \emph{A proof of the absolute purity conjecture (after Gabber)}, in: Algebraic geometry 2000, Azumino, Adv. Stud. Pure Math. \textbf{36} (2002), pp. 153--183.


\bibitem[Ga93]{Ga93}
O.\ Gabber,
\textit{An injectivity property for \'etale cohomology},
Compositio Math. \textbf{86} (1993), no.\ 1, 1--14.


\bibitem[Ga04]{Ga04}
O.\ Gabber,\emph{On purity for the Brauer group}, in: Arithmetic Algebraic Geometry, Oberwolfach Report \textbf{37}, (2004), 1975--1977

\bibitem[GMS03]{GMS03}
S.\ Garibaldi, A.\ Merkurjev, and J.-P.\ Serre,
\textit{Cohomological invariants in {G}alois cohomology},
University Lecture Series, vol. 28, AMS, Providence, R.I., 2003.


\bibitem[GS06]{GS06}
P.\ Gille and T.\ Szamuely, 
\textit{Central simple algebras and {G}alois cohomology},
Cambridge Studies in Advanced Mathematics, vol. 101,
Cambridge University Press, Cambridge, 2006.

\bibitem[Gro68]{Gro68}
A.\ Grothendieck, \emph{Le groupe de Brauer I, II, III}, Dix expos\'{e}s sur la cohomologie des sch\'{e}mas, North-Holland, Amsterdam, (1968). 

\bibitem[Ha67]{Ha67}
G.\ Harder, \emph{Halbeinfache Gruppenschemata über Dedekindringen}, Inventiones Math. \textbf{4} (1967) 165--191





\bibitem[Hi17]{Hi17}
H.\ Hironaka, \emph{Resolution of singularities in positive characteristics}, preprint (2017), available at \href{http://www.math.harvard.edu/~hironaka/pRes.pdf}{http://www.math.harvard.edu/~hironaka/pRes.pdf}

\bibitem[ILO14]{ILO14} 
L.\ Illusie, Y.\ Laszlo, and F.\ Orgogozo (eds.), 
\textit{Travaux de Gabber sur l'uniformisation locale et la
cohomologie \'etale des sch\'emas quasi-excellents}, 
Ast\'erisque, no.\ 363--364, Soci\'et\'e Math\'ematique de
France, Paris, 2014.


\bibitem[Izh91]{Izh91}
O.\ Izhboldin, \emph{On $p$-torsion in $K^M_{\ast}$ for fields of characteristic $p$}, in: Algebraic $K$-theory, Adv. Sov. math. \textbf{4} (1991), 129--144



\bibitem[Ka82]{Ka82}
K.\ Kato, \emph{Galois cohomology of complete discrete valuation fields}, Lecture Notes in Math. \textbf{67}, Springer-Verlag, Berlin (1982), 215--238

\bibitem[Ka86]{Ka86}
K.\ Kato, \emph{A Hasse principle for two dimensional global fields}, Journal f\"ur die reine und angewandte Mathematik \textbf{366} (1986), 142--180.

\bibitem[KMRT98]{KMRT98}
M.-A.\ Knus, A.\ Merkurjev, M.\ Rost, and J.-P.\ Tignol,
\textit{The Book of Involutions},
AMS Colloquium Publications, vol.\ 44, 1998.

\bibitem[Kr15]{Kr15}
J. Kr\"oker, \emph{A Macaulay2-Framework for finite field experiments for explicit and implicitly given ideals and parameter spaces}, available at \url{https://github.com/jakobkroeker/FiniteFieldExperiments.M2}.

\bibitem[KSC04]{KSC04}
J.\ Koll\'{a}r, K.E.\ Smith, A.\ Corti, \emph{Rational and Nearly Rational Varieties}, Cambridge Studies in Advanced Mathematics \textbf{92}, Cambridge University Press (2004)

\bibitem[Lieb08]{Lieb08}
M.\ Lieblich, \emph{Twisted sheaves and the period-index problem}, Compositio Mathematica \textbf{144} (1) (2008), 1--31

\bibitem[M2]{M2}
D.\ Grayson, M.\ Stillman, \emph{Macaulay2, a software system for research in algebraic geometry}, available at \url{http://www.math.uiuc.edu/Macaulay2/}
          

\bibitem[Mer15]{Mer15}
A. S. Merkurjev, 
\emph{Invariants of algebraic groups and retract rationality of classifying spaces}, preprint (2015) available at \href{http://www.math.ucla.edu/~merkurev/publicat.htm}{http://www.math.ucla.edu/~merkurev/publicat.htm}; 
to appear in AMS book Algebraic Groups: Structure and Actions.




\bibitem[Pi16]{Pi16}
A.\ Pirutka, 
\emph{Varieties that are not stably rational, zero-cycles
and unramified cohomology}, 
Algebraic Geometry (Salt Lake City, 2015), 
Proc. of Symposia in Pure Math., vol.\ 97, part II, pp.\ 459--484, 2018.  








\bibitem[TiWa15]{TiWa15}
J-P.\ Tignol and A.R. Wadworth, \emph{Value functions on simple algebras, and associated graded rings}, Springer Monographs in Math., Springer-Verlag (2015)

\bibitem[To16]{To16}
B.\ Totaro, \emph{Hypersurfaces that are not stably rational}, J. Amer. Math. Soc. \textbf{29} (2016), 883--891

\bibitem[Vac06]{Vac06}
M.\ Vacqui\'{e}, \emph{Valuations and local uniformization}, in: Singularity theory and its applications, Adv. Stud. Pure Math. \textbf{43} (2006), 477--527.

\bibitem[Voi15]{Voi15}
C.\ Voisin, \emph{Unirational threefolds with no universal codimension 2 cycle},  Invent.
Math. \textbf{201} (2015), 207--237.

\bibitem[Wei13]{Wei13}
Ch.\ Weibel, \emph{The K-book, An Introduction to Algebraic K-theory}, Graduate Studies in Math. \textbf{145}, AMS (2013)

\bibitem[Witt35]{Witt35}
E.\ Witt, 
\textit{\"Uber ein Gegenbeispiel zum Normensatz}, 
Math. Z. \textbf{39} (1935), no.\ 1, 462--467.

\bibitem[Z-S76]{Z-S76}
O.\ Zariski, P.\ Samuel, \emph{Commutative algebra, Volume II}, Graduate Texts in Mathematics \textbf{29}, Springer-Verlag (1976).
\end{thebibliography}
\end{document}